\documentclass[a4paper,11pt]{article}
\usepackage{amssymb,amsmath,amsthm,amsfonts}
\usepackage{mathrsfs}
\usepackage{multirow}
\usepackage{graphicx}
\usepackage{authblk}
\usepackage{indentfirst}
\usepackage{multicol}
\usepackage{url}
\usepackage{fancyhdr}
\usepackage[numbers]{natbib}
\usepackage[all]{xy}
\usepackage{mathtools}
\usepackage[english]{babel}
\usepackage{colortbl}
\usepackage{caption}
\usepackage{subcaption}
\usepackage{tikz}
\usepackage{float}
\usepackage{enumitem}
\usepackage{listings}
\usepackage{hyperref} % keep only this for links

\newtheorem{theorem}{Theorem}[section]
\newtheorem{proposition}[theorem]{Proposition}

\newtheorem{corollary}[theorem]{Corollary}

\newtheorem{definition}{Definition}[section]
\newtheorem{example}{Example}[section]

\newcommand{\im}{{\mathrm{im}\hspace{0.1em}}}

\definecolor{myred}{RGB}{180,0.00,0.00}
\definecolor{myblue}{RGB}{0.00,0.00,180}
\usepackage{xr}
\makeatletter
    \newcommand*{\addFileDependency}[1]{
    \typeout{(#1)}
    \@addtofilelist{#1}
    \IfFileExists{#1}{}{\typeout{No file #1.}}
    }
\makeatother

\newcommand{\rightcross}[1]{
  \begin{tikzpicture}[knot/.style={black}]
    \begin{scope}[xshift=5cm]
      \draw[line width=0.8pt] (-#1,#1)-- (#1,-#1);
      \draw[white,double=black,double distance=0.8pt,thick](-#1,-#1)-- (#1,#1);
    \end{scope}
  \end{tikzpicture}
}

\newcommand{\arightcross}[1]{
  \begin{tikzpicture}[knot/.style={black}]
    \begin{scope}[xshift=5cm]
      \draw[line width=0.8pt,->]  (#1,-#1)--(-#1,#1);
      \draw[white,double=black,double distance=0.8pt,thick](-#1,-#1)-- (#1,#1);
      \draw[line width=0.8pt,->] (#1*0.5,#1*0.5)--(#1,#1);
      \draw[line width=0.4pt](0,0)circle (#1*1.414);
    \end{scope}
  \end{tikzpicture}
}

\newcommand{\leftcross}[1]{
  \begin{tikzpicture}[knot/.style={black}]
    \begin{scope}[xshift=5cm]
      \draw[line width=0.8pt] (-#1,-#1)-- (#1,#1);
      \draw[white,double=black,double distance=0.8pt,thick](-#1,#1)-- (#1,-#1);
      \draw[line width=0.4pt](0,0)circle (#1*1.414);
    \end{scope}
  \end{tikzpicture}
}

\newcommand{\aleftcross}[1]{
  \begin{tikzpicture}[knot/.style={black}]
    \begin{scope}[xshift=5cm]
      \draw[line width=0.8pt,->] (-#1,-#1)-- (#1,#1);
      \draw[white,double=black,double distance=0.8pt,thick](-#1,#1)-- (#1,-#1);
      \draw[line width=0.8pt,->] (-#1*0.5,#1*0.5)-- (-#1,#1);
      \draw[line width=0.4pt](0,0)circle (#1*1.414);
    \end{scope}
  \end{tikzpicture}
}

\newcommand{\udarc}[1]{

  \begin{tikzpicture}
    \draw[line width=0.8pt] (0,0) arc (140:40:#1cm);
    \draw[line width=0.8pt] (0,#1*1.2) arc (220:320:#1cm);
     \draw[line width=0.4pt](#1*0.766,#1*0.6)circle (#1);
  \end{tikzpicture}
}

\newcommand{\cudarc}[1]{

  \begin{tikzpicture}
    \draw[line width=0.8pt] (0,0) arc (140:40:#1cm);
    \draw[line width=0.8pt] (0,#1*1.2) arc (220:320:#1cm);
    \draw[line width=0.4pt] (-#1*0.333, -#1*0.5) -- (-#1*0.333 , #1*1.7)--(#1*1.866 , #1*1.7)--(#1*1.866 , -#1*0.5)--(-#1*0.333, -#1*0.5);
  \end{tikzpicture}
}

\newcommand{\cudarcc}[1]{

  \begin{tikzpicture}
    \draw[line width=0.8pt] (0,0) arc (140:40:#1cm);
    \draw[line width=0.8pt] (0,#1*1.2) arc (220:320:#1cm);
    \draw[line width=0.4pt] (-#1*0.233, -#1*0.5) -- (-#1*0.233 , #1*1.7)--(#1*3.5 , #1*1.7)--(#1*3.5 , -#1*0.5)--(-#1*0.233, -#1*0.5);
    \draw[line width=0.8pt] (#1*1.532,0) arc[start angle=-140, end angle=140, radius=#1*0.933];
  \end{tikzpicture}
}

\newcommand{\longcircle}[1]{

  \begin{tikzpicture}
    \draw[line width=0.8pt] (0,0) arc (140:40:#1cm);
    \draw[line width=0.8pt] (0,#1*1.2) arc (220:320:#1cm);
    \draw[line width=0.4pt] (-#1*1.933, -#1*0.5) -- (-#1*1.933 , #1*1.7)--(#1*3.5 , #1*1.7)--(#1*3.5 , -#1*0.5)--(-#1*1.933, -#1*0.5);
    \draw[line width=0.8pt] (#1*1.532,0) arc[start angle=-140, end angle=140, radius=#1*0.933];
    \draw[line width=0.8pt] (0,0) arc[start angle= 320, end angle=40, radius=#1*0.933];
  \end{tikzpicture}
}

\newcommand{\rcrossingarc}[1]{

  \begin{tikzpicture}
    \draw[line width=0.8pt] (#1,-#1) arc (-135:135:#1 *1.414cm);
    \draw[line width=0.8pt](-#1*1.2,#1*1.2)-- (#1,-#1);
    \draw[white,double=black,double distance=0.8pt,thick] (-#1*1.2,-#1*1.2)-- (#1,#1);
  \end{tikzpicture}
}

\newcommand{\eight}[1]{

  \begin{tikzpicture}
    \draw[line width=0.8pt] (#1,-#1) arc (-135:135:#1 *1.414cm);
    \draw[line width=0.8pt] (-#1,#1) arc (45:315:#1 *1.414cm);
    \draw[line width=0.8pt](-#1,#1)-- (#1,-#1);
    \draw[white,double=black,double distance=0.8pt,thick] (-#1*1,-#1*1)-- (#1,#1);
  \end{tikzpicture}
}

\newcommand{\cudarco}[1]{

  \begin{tikzpicture}
    \draw[line width=0.4pt] (-#1*0.233, -#1*0.5) -- (-#1*0.233 , #1*1.7)--(#1*3.5 , #1*1.7)--(#1*3.5 , -#1*0.5)--(-#1*0.233, -#1*0.5);
    \draw[line width=0.8pt] (#1* 2.247,#1* 0.6)circle (#1*0.933);
    \draw[line width=0.8pt] (#1*0.5,0) arc[start angle=-40, end angle=40, radius=#1*0.933];
  \end{tikzpicture}
}

\newcommand{\ccircle}[1]{

  \begin{tikzpicture}
    \draw[line width=0.4pt] (-#1*1.233, -#1*0.5) -- (-#1*1.233 , #1*1.7)--(#1*3.5 , #1*1.7)--(#1*3.5 , -#1*0.5)--(-#1*1.233, -#1*0.5);
    \draw[line width=0.8pt] (#1* 2.247,#1* 0.6)circle (#1*0.933);
    \draw[line width=0.8pt] (#1* 0,#1* 0.6)circle (#1*0.933);
  \end{tikzpicture}
}

\newcommand{\saddle}[1]{

  \begin{tikzpicture}
    \draw[line width=0.8pt] (0,0) arc (140:40:#1cm);
    \draw[line width=0.8pt] (#1 *0.766,#1 *0.3572) -- (#1 *0.766, #1 *0.8428);
    \draw[line width=0.8pt] (0,#1*1.2) arc (220:320:#1cm);
    \draw[line width=0.4pt](#1*0.766,#1*0.6)circle (#1);
  \end{tikzpicture}
}

\newcommand{\lrarc}[1]{

  \begin{tikzpicture}
    \draw[line width=0.8pt] (#1*1.2,0) arc (130:230:#1cm);
    \draw[line width=0.8pt] (0,0) arc (50:-50:#1cm);
    \draw[line width=0.4pt](#1*0.6,-#1*0.766)circle (#1);
  \end{tikzpicture}
}

\newcommand{\clrarc}[1]{

  \begin{tikzpicture}
    \draw[line width=0.8pt] (#1*1.2,0) arc (130:230:#1cm);
    \draw[line width=0.8pt] (0,0) arc (50:-50:#1cm);
    \draw[line width=0.4pt] (-#1*0.5, -#1*1.866) -- (-#1*0.5 , #1*0.333)--(#1*1.7 , #1*0.333)--(#1*1.7 , -#1*1.866)--(-#1*0.5, -#1*1.866);
  \end{tikzpicture}
}

\newcommand{\hsaddle}[1]{
  \begin{tikzpicture}
    \draw[line width=0.8pt] (#1*1.2,0) arc (130:230:#1cm);
    \draw[line width=0.8pt] ( #1 * 0.3572, - #1 * 0.7660)-- (#1 * 0.2 + #1 * 0.6428, - #1 * 0.7660);
    \draw[line width=0.8pt] (0,0) arc (50:-50:#1cm);
    \draw[line width=0.4pt](#1*0.6,-#1*0.766)circle (#1);
  \end{tikzpicture}
}

\title{Khovanov homology of tangles:  algorithm and computation}

\author[1]{Li Shen$^{\ast}$}
\author[2,3]{Jian Liu$^{\ast}$}
\author[3,4,5]{Guo-Wei Wei\thanks{Corresponding author: weig@msu.edu}}
\affil[1]{NSF-Simons National Institute for Theory and Mathematics in Biology, Chicago IL.}
\affil[2]{Mathematical Science Research Center, Chongqing University of Technology, Chongqing 400054, China}
\affil[3]{Department of Mathematics, Michigan State University, MI 48824, USA}
\affil[4]{Department of Electrical and Computer Engineering, Michigan State University, MI 48824, USA}
\affil[5]{Department of Biochemistry and Molecular Biology, Michigan State University, MI 48824, USA}

\makeatletter
    \renewcommand*{\@fnsymbol}[1]{\ensuremath{\ifcase#1\or \dagger\or *\or *\or
   \mathsection\or \else\@ctrerr\fi}}
\makeatother
\date{}

\begin{document}
    % \linenumbers
    \maketitle
    \footnotetext[2]{Shen and Liu contributed equally to this work.
    The work of Liu was done during his two-year stay at Michigan State University.
    }

    \paragraph{Abstract}

Knot, link, and tangle theory is crucial in both mathematical theory and practical application, including quantum physics, molecular biology, and structural chemistry. Unlike knots and links, tangles impose more relaxed constraints, allowing the presence of arcs, which makes them particularly valuable for broader applications.
Although Khovanov homology for knots and links has been extensively studied, its computation for tangles remains largely unexplored.
In our recent work, we provide a topological quantum field theory (TQFT) construction for the Khovanov homology of tangles, offering a more concrete method for its computation.
The primary contribution of this work is a comprehensive approach to the computation of the Khovanov homology of tangles, offering both a detailed computation procedure and a practical guide for implementing algorithms through codes to facilitate the calculation. This contribution paves the way for further studies and applications of Khovanov homology in the context of tangles.

    \paragraph{Keywords}
     Tangle, Khovanov homology, cube state, planar code, algorithm.

\footnotetext[1]
{ {\bf 2020 Mathematics Subject Classification.}  	Primary  57K18; Secondary 57K10,  68W30.
}

\tableofcontents % insert contents

\section{Introduction}

Knot, link, and tangle are fundamental objects in geometric topology that study different types of curve entanglements in 3-dimensional (3D) space \cite{adams2004knot,rolfsen2003knots}.
These objects play a crucial role in understanding the topological properties of 3-manifolds and have applications in various fields, including quantum field theory in physics \cite{hsieh2012quantum,witten1989quantum}, protein folding in biology \cite{chou1974conformational,dill2012protein}, and molecular entanglements in chemistry \cite{fielden2017molecular,forgan2011chemical}.

Khovanov homology is a significant topological invariant for knots, links, and tangles, providing a classification for these objects and categorifying the Jones polynomial \cite{bar2002khovanov,bar2005khovanov,khovanov2000categorification,khovanov2002functor}. The study of Khovanov homology computations is crucial, both for advancing mathematical research and for providing practical tools for various applications.

Despite its theoretical significance, computing Khovanov homology by hand is computationally intensive due to the exponential growth of the number of states and the complexity of differential maps. Efficient algorithms and computer implementations are essential for calculating Khovanov homology for complex knots, links, and tangles.
Computational algorithms for computing Khovanov homology of knots and links, such as the topological quantum field theory (TQFT), are well-established \cite{bar2007fast,schmidhuber2025quantum,williams2008computations}.
Moreover, the requirement for data to be represented as closed curves imposes stringent constraints in many practical applications. In contrast, tangles impose less restrictive conditions compared to knots and links, since they permit the presence of arcs, offering a wider range of potential applications. Therefore, the development of algorithms and computational techniques for the Khovanov homology of tangles is an area of significant importance.

However, for years, there have been no feasible computational procedures to compute the Khovanov homology for tangles due to the lack of a TQFT construction and/or other computable algorithms for tangles. A recent work has closed this gap by proposing a TQFT construction for the Khovanov theory of tangles \cite{liu2024persistent}.  Specifically, a concrete functor that maps the category of tangles to the category of modules was given, which in principle enables the computation of tangle homology. However, the emphasis of this work was on the introduction of persistent Khovanov homology of tangles as an extension of earlier evolutionary persistent Khovanov homology \cite{shen2024evolutionary}. As such, detailed computational algorithms for the Khovanov theory of tangles remain a lack.

This work focuses on the computational aspects of Khovanov homology for tangles. We present methods, strategies, and algorithms to render these calculations feasible. The main contribution of this work is to provide a detailed method for calculating the Khovanov homology of tangles. This includes both a step-by-step guide for manual computation of Khovanov homology and a demonstration of how to use code to perform these calculations efficiently.

This paper is organized as follows. In the next section, we review some fundamental concepts that may be involved. Section \ref{section:encoding} presents the encoding representation of tangle diagrams. Section \ref{section:computation} provides the methods and properties for computing Khovanov homology, along with some computational examples. Finally, we explore the understanding of tangle's Khovanov homology from the perspective of algorithms and code, and provide practical examples.

\section{Preliminaries}

In this section, we briefly review the concepts of chain complexes, cochain complexes, homology, and cohomology, providing a foundation for understanding the Khovanov complex. In addition, we give a concise introduction to knots, links, and tangles.

\subsection{Chain complex and homology}

\subsubsection{Chain complex and cochain complex}

A chain complex provides an algebraic model of a geometric object. The grading of the chain complex corresponds to the dimension of the geometric elements. For example, in a tetrahedron, the interior is 3-dimensional, the faces are 2-dimensional, the edges are 1-dimensional, and the vertices are 0-dimensional; these correspond precisely to the generators in the chain complex.

For a geometric object such as a simplicial complex $K$, the boundary refers to its surface $\partial K$. The surface $\partial K$ is closed, which geometrically corresponds to the fact that the boundary of the boundary is empty, i.e.,
\[
\partial^2 K = \partial(\partial K) = \emptyset.
\]

Let $\mathbb{K}$ be a coefficient ring or field, for example the real numbers $\mathbb{R}$, the rational numbers $\mathbb{Q}$, or the integers $\mathbb{Z}$. A \textbf{chain complex} $(C_{\ast},d_{\ast})=(C_{k}, d_{k})_{k\geq 0}$ over $\mathbb{K}$ is a graded $\mathbb{K}$-module
\[
C = \bigoplus_{k\ge 0} C_k
\]
equipped with boundary maps (or differentials) $d_k : C_k \to C_{k-1}$ satisfying
\[
d_{k-1} \circ d_k = 0 \quad \text{for all } k\ge 1, \quad \text{and} \quad d_0 = 0.
\]
The condition $d_{k-1} \circ d_k = 0$ corresponds algebraically to the geometric intuition that the boundary of a boundary is empty.

\begin{example}
Consider a 2-simplex (triangle) $\triangle ABC$. The associated chain complex over a $\mathbb{Z}$ is
\[
C_2 \xrightarrow{d_2} C_1 \xrightarrow{d_1} C_0 \xrightarrow{d_0} 0,
\]
where the chain groups are
\[
C_2 = \langle [ABC] \rangle, \quad
C_1 = \langle [AB], [BC], [CA] \rangle, \quad
C_0 = \langle [A], [B], [C] \rangle,
\]
and the boundary maps are
\[
\begin{aligned}
d_2([ABC]) &= [BC] - [AC] + [AB],\\
d_1([AB]) &= [B] - [A], \quad d_1([BC]) = [C] - [B], \quad d_1([CA]) = [A] - [C],\\
d_0 &= 0.
\end{aligned}
\]

Geometrically, $C_2$ corresponds to the triangular face, $C_1$ to the edges, and $C_0$ to the vertices. The condition $d_{k-1}\circ d_k = 0$ reflects the geometric intuition that ``the boundary of a boundary is empty.''
\end{example}

Dually, a \textbf{cochain complex} $(C^{\ast},d^{\ast})=(C^{k}, d^{k})_{k\geq 0}$ is a graded $\mathbb{K}$-module
\[
  C = \bigoplus_{k\ge 0} C_k
\]
equipped with coboundary maps $d^k : C^k \to C^{k+1}$ satisfying
\[
  d^{k+1} \circ d^k = 0 \quad  \text{for all } k\ge 0.
\]

There are many constructions of chain complexes and cochain complexes that provide algebraic models of geometric or topological objects. For instance, one can consider the following examples:

\begin{itemize}
    \item The \emph{chain complex of a simplicial complex}, where the chain groups are generated by simplices of various dimensions and the boundary maps encode their faces.
    \item The \emph{de Rham complex of a smooth manifold}, where the cochain groups consist of differential forms and the coboundary operator is given by the exterior derivative.
    \item The \emph{Khovanov complex of a knot or tangle}, where the chain groups are generated by smoothing states of the diagram and the differential is defined via saddle cobordisms.
    \item The \emph{cellular chain complex of a CW complex}, where the chain groups are generated by cells and the boundary maps reflect the attaching maps between cells.
\end{itemize}

Each of these constructions provides an algebraic characterization of the underlying space or object, allowing for computations of homology or cohomology that capture topological invariants.

\subsubsection{Homology}

Homology and cohomology are among the most fundamental topological invariants used to characterize a topological space or a geometric object. They encode intrinsic geometric properties of the underlying space.

Let $f: V \to W$ be a linear map between linear spaces over a field $\mathbb{K}$. The \textbf{kernel} of $f$ is defined by
\[
\ker f = \{ x \in V \mid f(x) = 0 \}.
\]
The \textbf{image} of $f$ is defined by
\[
\im f = \{ f(x) \in W \mid x \in V \}.
\]
It is clear that $\ker f$ is a subspace of $V$, while $\mathrm{im}\, f$ is a subspace of $W$.

Let $(C_{\ast},d_{\ast}) = (C_k, d_k)_{k \ge 0}$ be a chain complex over a field $\mathbb{K}$. We have a sequence
\[
\xymatrix{
\cdots \ar[r]^{d_{k+1}} & C_k \ar[r]^{d_k} & C_{k-1} \ar[r]^{d_{k-1}} & \cdots \ar[r]^{d_1} & C_0 \ar[r]^{d_0} & 0.
}
\]
The \textbf{homology} of $(C_k, d_k)_{k \ge 0}$ is defined by
\[
H_k(C,d) = \frac{\ker d_k}{\mathrm{im}\, d_{k+1}}, \quad k \ge 0.
\]

Similarly, for a cochain complex $(C^\ast, d^\ast) = (C^k, d^k)_{k \ge 0}$ over $\mathbb{K}$, the corresponding \textbf{cohomology} is defined by
\[
H^k(C^\ast, d^\ast) = \frac{\ker d^k}{\mathrm{im}\, d^{k-1}}, \quad k \ge 0.
\]

It is worth noting that the grading of a chain complex and its homology can be extended to the integers $\mathbb{Z}$. In this case, the chain groups $C_k$ and the homology groups $H_k(C)$ are defined for all $k \in \mathbb{Z}$, and the grading may take negative values.

\subsection{Tangle and crossing}

\subsubsection{Knots, links, and tangles}
A \textbf{knot} is an embedding of the circle into the three-dimensional Euclidean space $\mathbb{R}^3$. Intuitively, a knot can be regarded as a closed curve in $\mathbb{R}^3$ without self-intersections or singular points. Such structures frequently appear in physics, chemistry, biology, and materials science. However, the number of distinct knot types is relatively limited, which restricts the scope of practical applications involving knot-type data.

In contrast, a \textbf{link} is a disjoint embedding of finitely many circles into $\mathbb{R}^3$, subject to looser constraints. Moreover, a \textbf{tangle}, defined as an embedding of finitely many arcs and circles into a three-dimensional ball, imposes even fewer restrictions. As a result, the variety of tangle types is considerably richer, suggesting a broader range of potential application scenarios.

A \textbf{tangle diagram} is defined as a projection
\[
T \longrightarrow B^{2}
\]
of a tangle \(T\) onto a maximal disk \(B^{2} \subset B^{3}\), which is injective except at finitely many \textbf{crossing points}.
\begin{figure}[h]
  \centering
  \includegraphics[width=0.2\textwidth]{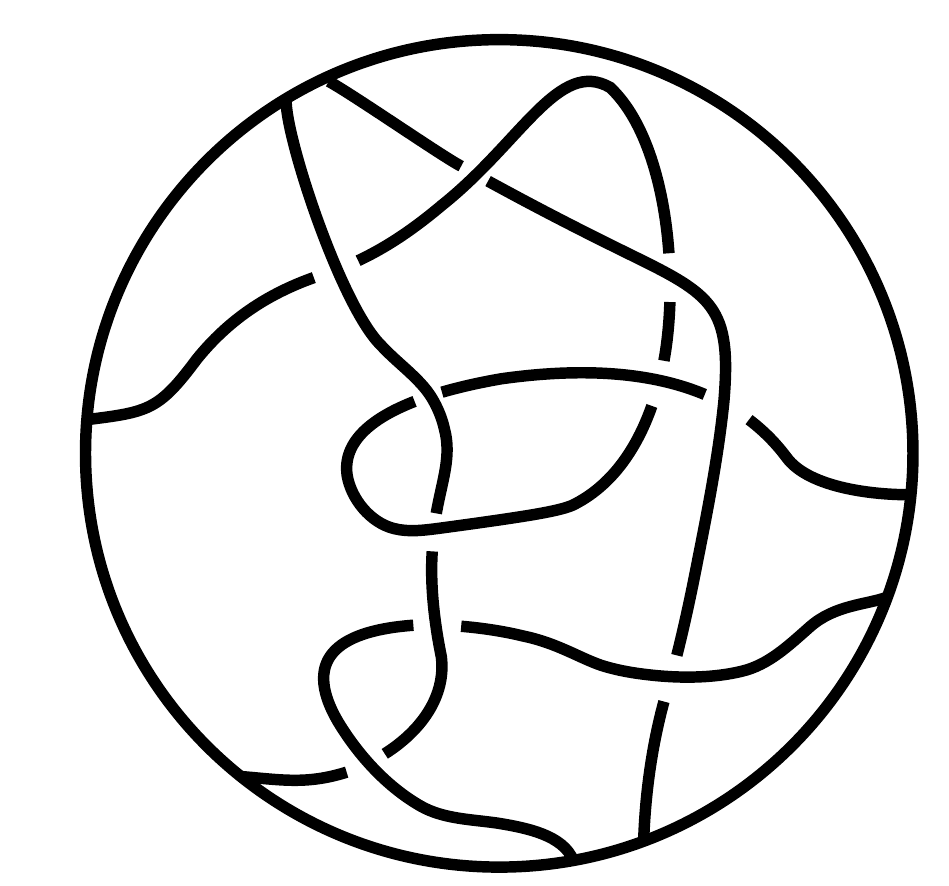}
  \caption{Illustration of a tangle diagram.}\label{figure:tangle_diagram}
\end{figure}
Each crossing point corresponds to the projection of exactly two distinct points of the tangle. This notion generalizes the concepts of knot diagrams and link diagrams.

A \textbf{Reidemeister move} of a tangle diagram,  as shown in Figure \ref{figure:R_move}, is a local transformation consisting of one of the following three types:
\begin{enumerate}[label=(R\arabic*), labelindent=2em, leftmargin=*, labelsep=0.6em]
  \item twisting or untwisting a single loop;
  \item adding or removing two crossings;
  \item sliding a strand over or under a crossing.
\end{enumerate}
\begin{figure}[h]
  \centering
  \includegraphics[width=0.6\textwidth]{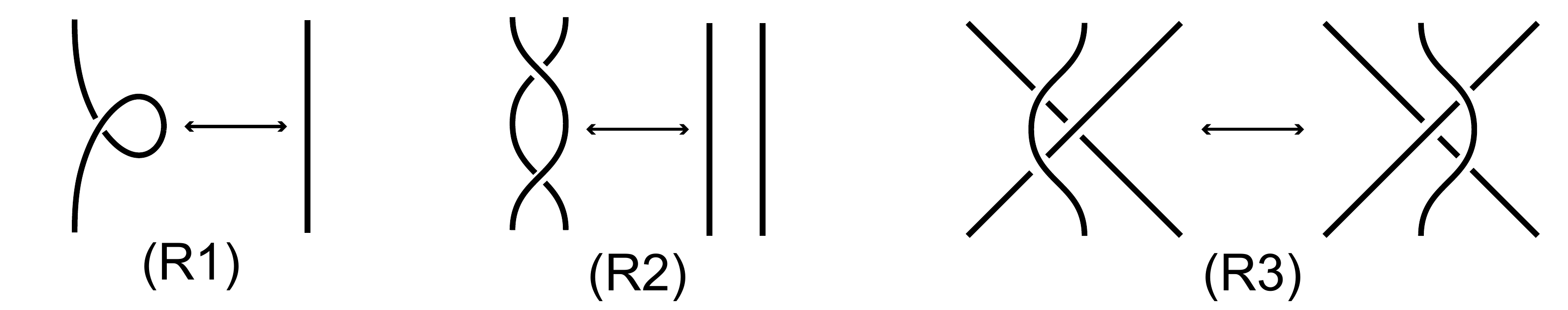}
  \caption{Illustration of the three types of Reidemeister moves.}\label{figure:R_move}
\end{figure}
Two tangle diagrams are said to be \textbf{equivalent} if they are related by a finite sequence of Reidemeister moves.

\subsubsection{Crossings and their types}

In a tangle diagram, a crossing point is referred to as a \textbf{crossing}. For a given tangle diagram, a crossing of the form $\raisebox{-0.15cm}{\rightcross{0.2}}$ is called an \textbf{overcrossing}, while one of the form $\raisebox{-0.15cm}{\leftcross{0.2}}$ is called an \textbf{undercrossing}. Each crossing admits two smoothing resolutions:
\[
\raisebox{-0.15cm}{\rightcross{0.2}} \ \Rightarrow\  \raisebox{-0.15cm}{\udarc{0.3}} \ +\ \raisebox{-0.15cm}{\lrarc{0.3}},
\qquad
\raisebox{-0.15cm}{\leftcross{0.2}} \ \Rightarrow\  \raisebox{-0.15cm}{\udarc{0.3}} \ +\ \raisebox{-0.15cm}{\lrarc{0.3}}.
\]
Here, $\raisebox{-0.15cm}{\udarc{0.3}}$, called the \textbf{0-smoothing}, is obtained by locally replacing the crossing with two opposing arcs arranged vertically (one above the other), while $\raisebox{-0.15cm}{\lrarc{0.3}}$, called the \textbf{1-smoothing}, is obtained by replacing it with two opposing arcs arranged horizontally (one to the left and one to the right). In this work, all 0-smoothings and 1-smoothings are applied to undercrossings $\raisebox{-0.15cm}{\leftcross{0.2}}$.

A tangle is said to be \textbf{oriented} if each of its strands is assigned a direction, typically indicated by arrows along the strands in its diagram. An orientation allows us to distinguish between the start and end of each strand. Given a fixed orientation, each crossing in a tangle can take one of two forms: a \textbf{right-handed crossing} of the form $\raisebox{-0.15cm}{\arightcross{0.2}}$ or a \textbf{left-handed crossing} of the form $\raisebox{-0.15cm}{\aleftcross{0.2}}$. We always assign the symbol $+$ to right-handed crossings and the symbol $-$ to left-handed crossings. We also commonly denote the number of right-handed crossings by $n_{+}$, and the number of left-handed crossings by $n_{-}$.

Given a tangle $T$ with $n$ crossings, each crossing can be \textbf{smoothed} in two possible ways, yielding $2^{n}$ distinct \textbf{states of smoothings}. These $2^{n}$ states naturally form the vertices of an $n$-dimensional cube $\{0,1\}^{n}$. Each vertex of the cube corresponds to a particular state, which can be denoted by
\begin{equation*}
  s = (s_{1}, s_{2}, \dots, s_{n}) \in \{0,1\}^{n}.
\end{equation*}
The state $s$ corresponds to a tangle $T_s$, which consists of a disjoint union of circles and arcs.

\section{Encodings of tangles}\label{section:encoding}

\subsection*{Gauss Code for Tangles}

We introduce the \emph{Gauss code for tangles}, a combinatorial representation that applies uniformly to both knots/links (closed components) and tangles (which may include open strands).

A tangle diagram $T$ with $n$ components is assigned a Gauss code consisting of two parts:
\[
\bigl[ C_1, C_2, \dots, C_n \bigr], \quad (t_1, t_2, \dots, t_n),
\]
where each $C_i$ is a sequence of signed integers encoding the crossings encountered along the $i$-th component, and $t_i \in \{\texttt{o}, \texttt{c}\}$ specifies whether that component is \texttt{open} or \texttt{closed}, respectively.

The construction proceeds as follows:
\begin{enumerate}
  \item Traverse each component in a fixed orientation --- from one endpoint for open strands, or arbitrarily for closed loops;
  \item Assign global crossing labels ($1,2,3,\dots$) in the order of first appearance;
  \item When a crossing is encountered, record its label with a sign: positive for overcrossings, negative for undercrossings;
  \item Repeat for all components to obtain the tuple of sequences $\{C_i\}$.
\end{enumerate}

By definition, each crossing label appears \textbf{exactly twice} across all sequences. Optionally, one may refine the notation by attaching \emph{handedness} information (e.g., $+1L$, $-2R$), producing the \emph{extended Gauss code}.

\medskip

\begin{example}
\begin{figure}[h]
  \centering
  \includegraphics[width=0.4\textwidth]{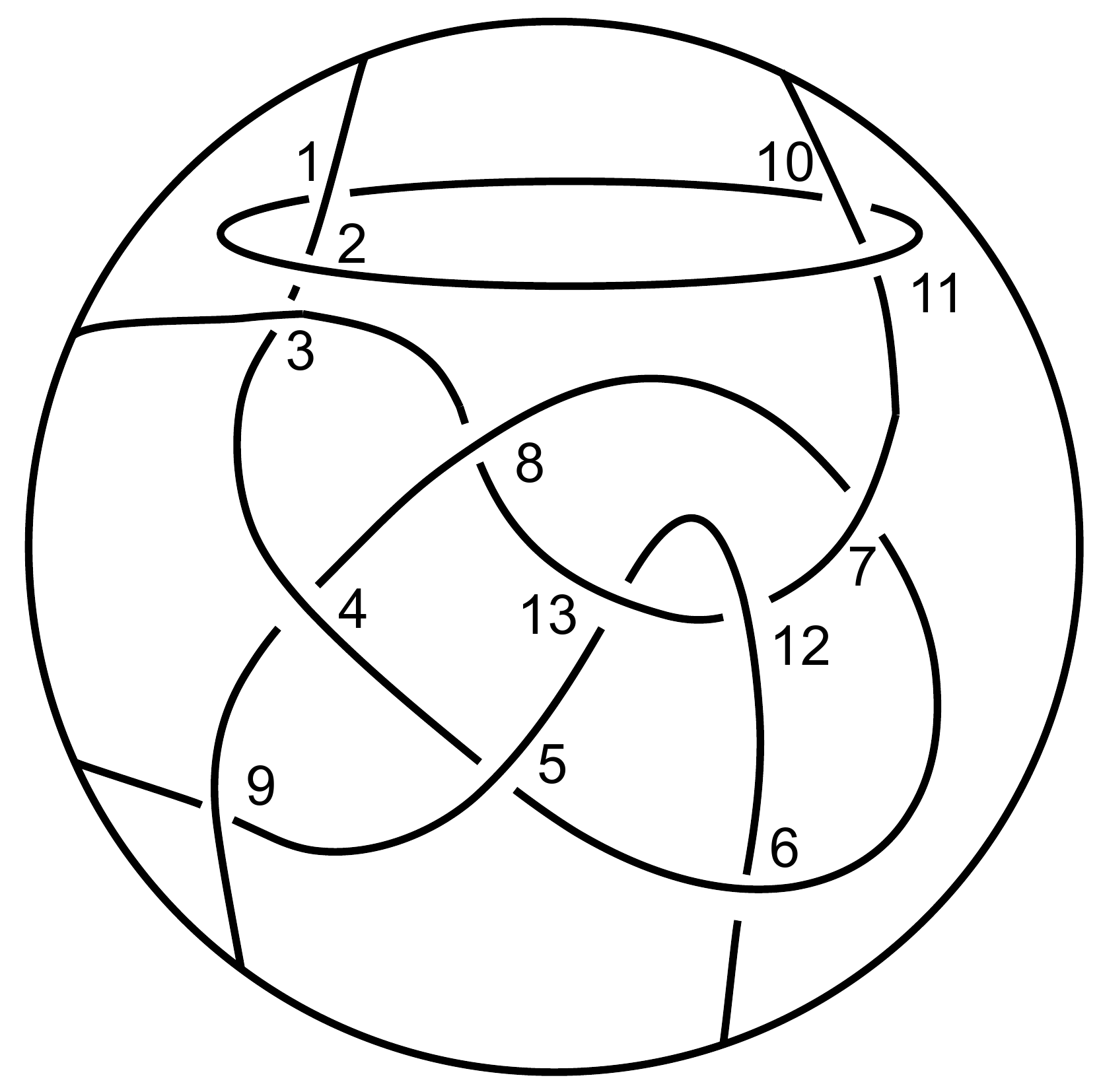}
  \caption{Illustration of Gauss code.}\label{figure:gauss_code}
\end{figure}

Consider a tangle with three open strands and one closed loop as shown in Figure \ref{figure:gauss_code}. Its Gauss code can be written as

\[
\begin{aligned}
&[[+1,-2,-3,+4,-5,+6,-7,+8,-4,+9], \\
&[+10,-11,+7,-12,+13,-8,+3], \\
&[-9,+5,-13,+12,-6], \\
&[-1,+2,+11,-10]]
\end{aligned}
\qquad
\begin{aligned}
&(o,o,o,c)
\end{aligned}
\]
\end{example}
\medskip

\subsection{Planar code}

The \textbf{planar code}, also known as the \textit{planar diagram code} or \textit{PD code}, is another way to encode a knot diagram. Unlike the Gauss code, which is based on a traversal of the knot and records crossings in the order they are encountered, the planar code captures the \textit{local structure} at each crossing by listing the four strands involved in a fixed order.

Each crossing is represented by a 4-tuple of integers:
\[
[a, b, c, d],
\]
where $a, b, c, d$ are the labels of the strands touching the crossing.

To construct a planar code:
\begin{itemize}
  \item Assign a unique integer to each strand of the knot diagram (i.e., each segment between two crossings).
  \item For each crossing, record the four adjacent strands, beginning with the incoming understrand and then listing them counterclockwise.
\end{itemize}

\noindent\textbf{Extension to Tangles.} \\
When encoding tangles, which may include both closed loops and open strands, the planar code is adapted by allowing certain strands to terminate at the boundary of the diagram. To mark such open-ended strands, we annotate their labels with a vertical bar \texttt{|}. That is, if strand label $a$ includes an endpoint, we write it as \texttt{|}$a$ or $a$\texttt{|}.

\begin{example}
\begin{figure}[h]
  \centering
  \includegraphics[width=0.15\textwidth]{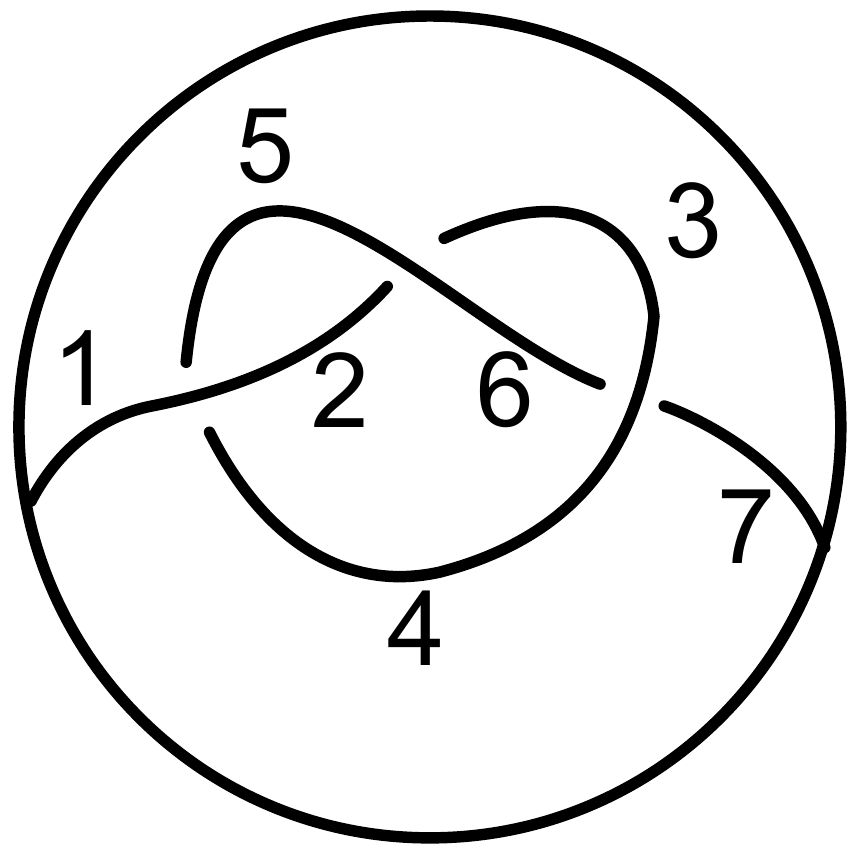}
\caption{Illustration of plannar code.}\label{figure:plannarcode}
\end{figure}

Consider a tangle with three crossings as shown in Figure \ref{figure:plannarcode}. Its planar code can be written as

\[
[[|1,4,2,5], [2,5,3,5], [6,4,7|,3]]
\]
\end{example}

\section{Computing Khovanov homology for tangles}\label{section:computation}

In this section, we present the computational procedure for determining the Khovanov homology of a tangle in the category of vector spaces.

\subsection{Khovanov homology of tangles}

The Khovanov complex is a chain complex constructed from a knot or tangle diagram by analyzing all possible smoothing states of its crossings. It categorifies the Jones polynomial by associating graded chain groups to diagrams and defining differentials between them. Each smoothing state corresponds to a basis element in the complex, and the differentials are constructed based on local transformations between states.

\subsubsection{State cube}
Let $T$ be a tangle with $n$ crossings. This gives rise to an $n$-cube $\{0,1\}^{n}$, where each vertex $s$ corresponds to a smoothing tangle $T_s$. We collect all tangles in the $n$-cube that are at the same distance from the origin as a single object. Specifically, we define
\begin{equation*}
  [[T]]^{k} = \bigoplus_{\ell(s) = k} T_s,
\end{equation*}
where $\ell(s)$ denotes the number of $1$'s in the state $s$, that is, the distance from $s$ to the origin.

We represent an edge of the $n$-cube by
\begin{equation*}
  \xi = (\xi_1, \dots, \xi_{i-1}, \star, \xi_{i+1}, \dots, \xi_{n}) \in \{0, 1, \star\}^{n},
\end{equation*}
where $\xi_j \in \{0,1\}$ and $\star$ indicates the connecting map $0 \to 1$. More precisely, the edge
\begin{equation*}
  (\xi_1, \dots, \xi_{i-1}, \star, \xi_{i+1}, \dots, \xi_{n})
\end{equation*}
connects the vertices
\begin{equation*}
  s = (\xi_1, \dots, \xi_{i-1}, 0, \xi_{i+1}, \dots, \xi_{n}) \quad \text{and} \quad s' = (\xi_1, \dots, \xi_{i-1}, 1, \xi_{i+1}, \dots, \xi_{n}).
\end{equation*}
Consequently, $\xi$ determines a morphism
\begin{equation*}
  d_{\xi} : T_s \to T_{s'}.
\end{equation*}

For a given state $s$, each position with value $0$ corresponds to an edge connecting $s$ to another state $s'$ where that position is $1$ and all other positions remain the same. Thus, there are $n - \ell(s)$ such edges $\xi$, each giving rise to a map
\begin{equation*}
  d_{\xi} : T_s \to T_{s'}.
\end{equation*}
If $T_s$ and $T_{s'}$ are not connected by $\xi$, we set $d_{\xi} = 0$ on $T_s$.

\subsubsection{Description of $d_{\xi}$}

%\begin{highlightbox}

Note that the map $d_{\xi}$ affects only the smoothing at the crossing corresponding to the $\star$ position. Locally, there are only three possible cases:

\begin{enumerate}[label=(\roman*)]
  \item The crossing involves two intersecting arcs, for example, $\raisebox{-0.2cm}{\rightcross{0.3}}$. In this case, $d_{\xi}$ transforms the two arcs as
  \begin{equation*}
    \raisebox{-0.15cm}{\hsaddle{0.3}}:\raisebox{-0.15cm}{\clrarc{0.3}}\to \raisebox{-0.15cm}{\cudarc{0.3}}.
  \end{equation*}

  \item The crossing arises from a single arc crossing over itself, for example, $\raisebox{-0.2cm}{\rcrossingarc{0.18}}$. Here, $d_{\xi}$ has two possible local transformations
  \begin{equation*}
    \raisebox{-0.15cm}{\saddle{0.3}}:\raisebox{-0.15cm}{\cudarcc{0.3}}\to \raisebox{-0.15cm}{\cudarco{0.3}}
  \end{equation*}
  and
  \begin{equation*}
    \raisebox{-0.15cm}{\hsaddle{0.3}}:\raisebox{-0.15cm}{\cudarco{0.3}}\to \raisebox{-0.15cm}{\cudarcc{0.3}}.
  \end{equation*}

  \item The crossing comes from a closed loops, for example, $\raisebox{-0.2cm}{\eight{0.18}}$. In this case, $d_{\xi}$ also has two possible local transformations
  \begin{equation*}
    \raisebox{-0.15cm}{\saddle{0.3}}:\raisebox{-0.15cm}{\longcircle{0.3}}\to \raisebox{-0.15cm}{\ccircle{0.3}}
  \end{equation*}
  and
  \begin{equation*}
    \raisebox{-0.15cm}{\hsaddle{0.3}}:\raisebox{-0.15cm}{\ccircle{0.3}}\to \raisebox{-0.15cm}{\longcircle{0.3}}.
  \end{equation*}
\end{enumerate}

Here, the boxed regions emphasize the local part of the tangle being modified, while the rest of the tangle remains unchanged.

%\end{highlightbox}

Based on the above construction, we define
\begin{equation*}
  d^k = \sum_{|\xi|=k} (-1)^{\mathrm{sgn}(\xi)} d_{\xi} : [[T]]^{k} \to [[T]]^{k+1},
\end{equation*}
where $|\xi|$ denote the number of $1$s in $\xi$, and the sign $\mathrm{sgn}(\xi)$ is determined by the number of $1$s in $\xi$ appearing before the first $\star$.

Consequently, we obtain a sequence of morphisms between the direct sum collections of tangle diagrams
\begin{equation*}
  \xymatrix{
  [[T]]^{0}\ar[r]^{d^{0}} & [[T]]^{1}\ar[r] & \cdots \ar[r] & [[T]]^{n-1}\ar[r]^{d^{n-1}} & [[T]]^{n}.
  }
\end{equation*}

It is straightforward to verify that $d^{k+1} \circ d^{k} = 0$ for $k = 0,1,\dots,n-1$.

\subsubsection{The construction $\mathcal{G}$}

To obtain a cochain complex in the usual sense, we proceed with the following construction.

Recall that $T_s$ consists of a disjoint union of circles and arcs. We now define the functor $\mathcal{G}$. Let $\mathbb{K}$ be a coefficient ring or field, for instance, the integers $\mathbb{Z}$, a finite field $\mathbb{Z}/p$, or the real numbers $\mathbb{R}$.

Let $V$ be the $\mathbb{K}$-vector space spanned by the elements $v_{+}$ and $v_{-}$, and let $W$ be the $\mathbb{K}$-vector space spanned by the element $w$. For a tangle $T$ consisting of $r$ circles and $t$ arcs, we define
\begin{equation*}
  \mathcal{G}(T) = \underbrace{W \otimes_{\mathbb{K}} \cdots \otimes_{\mathbb{K}} W}_{t} \otimes \underbrace{V \otimes_{\mathbb{K}} \cdots \otimes_{\mathbb{K}} V}_{r}.
\end{equation*}
Note that $V = \mathbb{K}\{v_{+}, v_{-}\}$ and $W = \mathbb{K}\{w\}$. Therefore, $\mathcal{G}(T)$ is a $2^{r}$-dimensional $\mathbb{K}$-vector space.

The functor $\mathcal{G}$ is defined as follows:
\begin{equation*}
  \begin{split}
      &  \mathcal{G}(\raisebox{-0.15cm}{\hsaddle{0.3}}:\raisebox{-0.15cm}{\clrarc{0.3}}\to \raisebox{-0.15cm}{\cudarc{0.3}}): W\otimes W \to W\otimes W, \quad w\otimes w \mapsto 0,\\
      &  \mathcal{G}(\raisebox{-0.15cm}{\saddle{0.3}}:\raisebox{-0.15cm}{\cudarcc{0.3}}\to \raisebox{-0.15cm}{\cudarco{0.3}}): W \to W\otimes V, \quad w \mapsto w\otimes v_{-},\\
      &  \mathcal{G}(\raisebox{-0.15cm}{\hsaddle{0.3}}:\raisebox{-0.15cm}{\cudarco{0.3}}\to \raisebox{-0.15cm}{\cudarcc{0.3}} ): W\otimes V \to W, \quad \left\{
    \begin{array}{ll}
      w \otimes v_{+} \mapsto w, \\
      w \otimes v_{-} \mapsto 0,
    \end{array}
  \right.\\
      &\mathcal{G}(\raisebox{-0.15cm}{\saddle{0.3}}:\raisebox{-0.15cm}{\longcircle{0.3}}\to \raisebox{-0.15cm}{\ccircle{0.3}}): V\otimes V\to V,\quad \left\{
                                                                                                                                                        \begin{array}{ll}
                                                                                                                                                          v_{+}\mapsto v_{+}\otimes v_{-}+v_{-}\otimes v_{+},  \\
                                                                                                                                                          v_{-}\mapsto v_{-}\otimes v_{-},
                                                                                                                                                        \end{array}
                                                                                                                                                      \right.
      \\
      &\mathcal{G}(\raisebox{-0.15cm}{\hsaddle{0.3}}: \raisebox{-0.15cm}{\ccircle{0.3}}\to \raisebox{-0.15cm}{\longcircle{0.3}}: V\to V\otimes V,\quad \left\{
                                                                                                                                                         \begin{array}{ll}
                                                                                                                                                           v_{+}\otimes v_{+}\mapsto v_{+}, \\
                                                                                                                                                           v_{+}\otimes v_{-}\mapsto v_{-}, \\
                                                                                                                                                           v_{-}\otimes v_{+}\mapsto v_{-}, \\
                                                                                                                                                           v_{-}\otimes v_{-}\mapsto 0.
                                                                                                                                                         \end{array}
                                                                                                                                                       \right.
  \end{split}
\end{equation*}

\subsubsection{Khovanov homology}

In our work, the construction of Khovanov homology is heavily based on the functor $\mathcal{G}$, which assigns linear spaces to smoothing tangles, thus providing a practical framework for computing homology. The functor $\mathcal{G}$ can be regarded as a generalization of constructions from topological quantum field theory for knots and links \cite{khovanov2000categorification}.

Let us denote
\begin{equation*}
  Kh^{k}(T) = \mathcal{G}([[T]]^{k+n_{-}}), \quad \text{and} \quad d_{T}^{k} = \mathcal{G}(d^{k+n_{-}}).
\end{equation*}
This gives rise to the sequence
\begin{equation*}
  \xymatrix{
  Kh^{-n_{-}}(T)\ar[r]^{d^{-n_{-}}_{T}} & Kh^{1-n_{-}}(T)\ar[r] & \cdots \ar[r] & Kh^{n_{+}-1}(T)\ar[r]^{d^{n_{+}-1}_{T}} & Kh^{n_{+}}(T).
  }
\end{equation*}

\begin{proposition}[{\cite[Proposition 3.2]{liu2024persistent}}]
The construction $(Kh^{\ast}(T),d_{T}^{\ast})$ is a cochain complex.
\end{proposition}

\begin{definition}
The Khovanov homology of a tangle $T$ is defined as the cohomology of the cochain complex $(Kh^{\ast}(T), d_{T}^{\ast})$. More precisely,
\[
   H^{k}(T) = H^{k}(Kh^{\ast}(T), d_{T}^{\ast}), \quad k \in \mathbb{Z}.
\]
\end{definition}

\subsubsection{Quantum grading}
Recall that $V = \mathbb{K}\{v_{+}, v_{-}\}$ and $W = \mathbb{K}\{w\}$. We define the \textbf{quantum grading} $\theta$ by
\begin{equation*}
  \theta(v_{+}) = 1, \quad \theta(v_{-}) = \theta(w) = -1,
\end{equation*}
and extend it to tensor products via
\begin{equation*}
  \theta(x \otimes y) = \theta(x) + \theta(y).
\end{equation*}
Consequently, any generator obtained as a tensor product of $v_{+}$, $v_{-}$, and $w$ has a well-defined quantum grading. This quantum grading provides an additional filtration that enriches the structure of the Khovanov homology. This quantum grading provides an additional filtration that enriches the structure of the Khovanov homology. Consequently, it can offer extra features for data applications and enhance their practical potential.

Furthermore, for a generator $[x]$ of the Khovanov homology, we define its quantum grading by
\begin{equation*}
  \Phi([x]) = p(x) + n_{+} - n_{-} + \theta(x),
\end{equation*}
where $p(x)$ denotes the homological degree corresponding to $[x]$.

\subsection{Theorems for computation}

\subsubsection{Orientation}

We have established the basic notions of the Khovanov homology of tangles. We now describe its computation. In the computational process, determining whether a crossing in a tangle is right-handed or left-handed requires specifying an orientation of the tangle.

For knots, the Khovanov homology is invariant under changes of orientation. However, for links and tangles, this invariance does not necessarily hold.

\begin{definition}
Let $T$ be a tangle.
Two orientations on $T$ are said to be \textbf{equivalent} if, for every crossing, the type of the crossing remains unchanged; that is, the sign (positive or negative) of each crossing is preserved.
\end{definition}

\begin{definition}
Let $T$ be a tangle. The \textbf{sign type} of $T$ is a function $\sigma: \mathcal{X}(T)\to \{+,-\}$ from the set of crossings to the set $\{+,-\}$.
\end{definition}

For simplicity, once the ordering of the crossings is fixed, we may regard \(\sigma\) as an element of \(\{+,-\}^n\), where \(n\) denotes the number of crossings of \(T\). We denote by \(n_{+}(\sigma)\) and \(n_{-}(\sigma)\) the numbers of \(+\) and \(-\) signs in \(\sigma\), respectively.
If $T$ is a knot, for any two sign types $\sigma$ and $\tau$ of $T$, we have
\begin{equation*}
  \sigma(T) = \tau (T).
\end{equation*}
For links or tangles, the sign types are not fixed.

\begin{example}
For the trefoil knot, different orientations do not change its sign markings, whereas for the Hopf link, different orientations may alter the sign markings, as illustrated in Figure \ref{figure:orientation}.
\begin{figure}[h]
  \centering
  \includegraphics[width=0.25\textwidth]{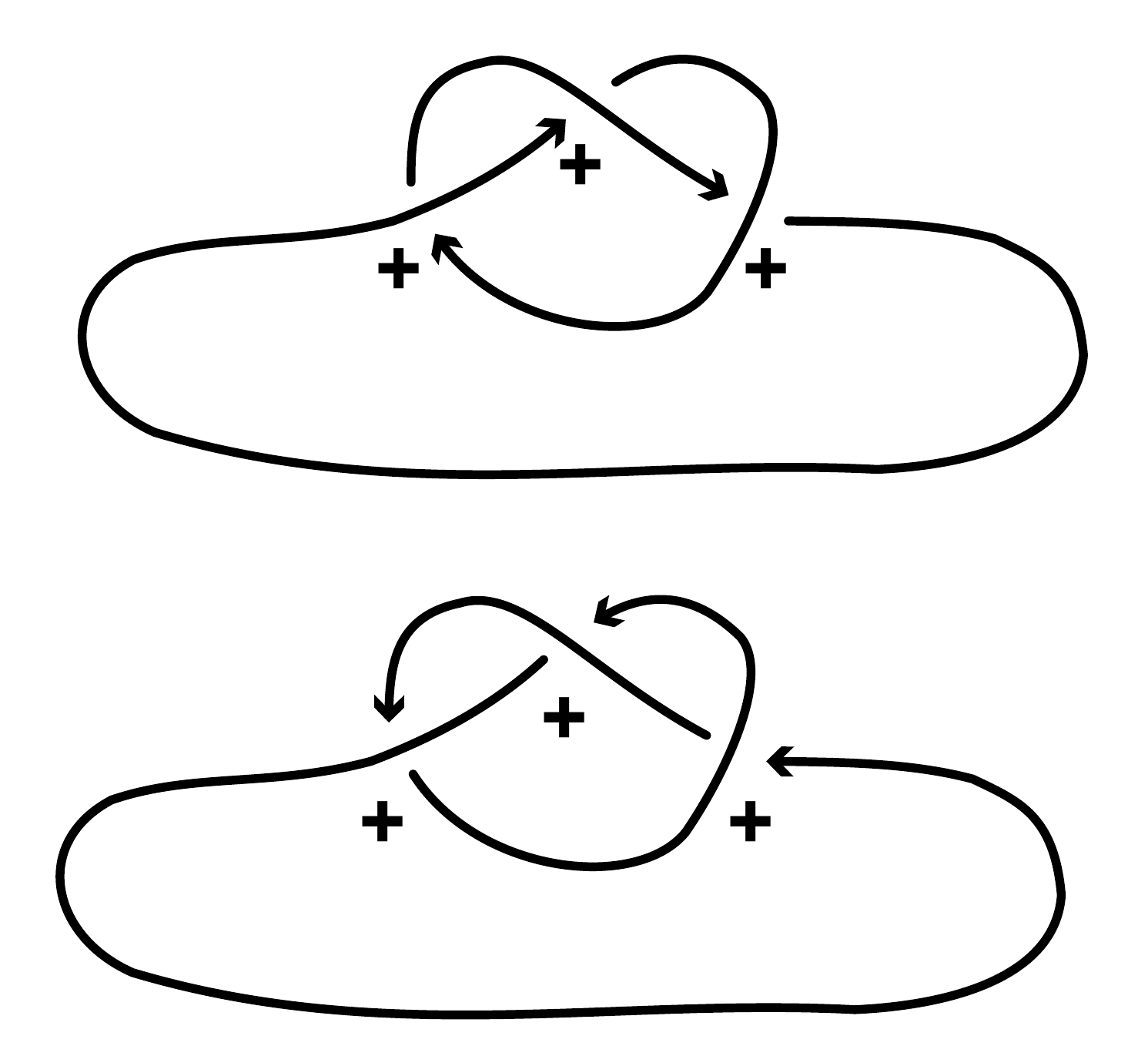}\qquad
  \includegraphics[width=0.18\textwidth]{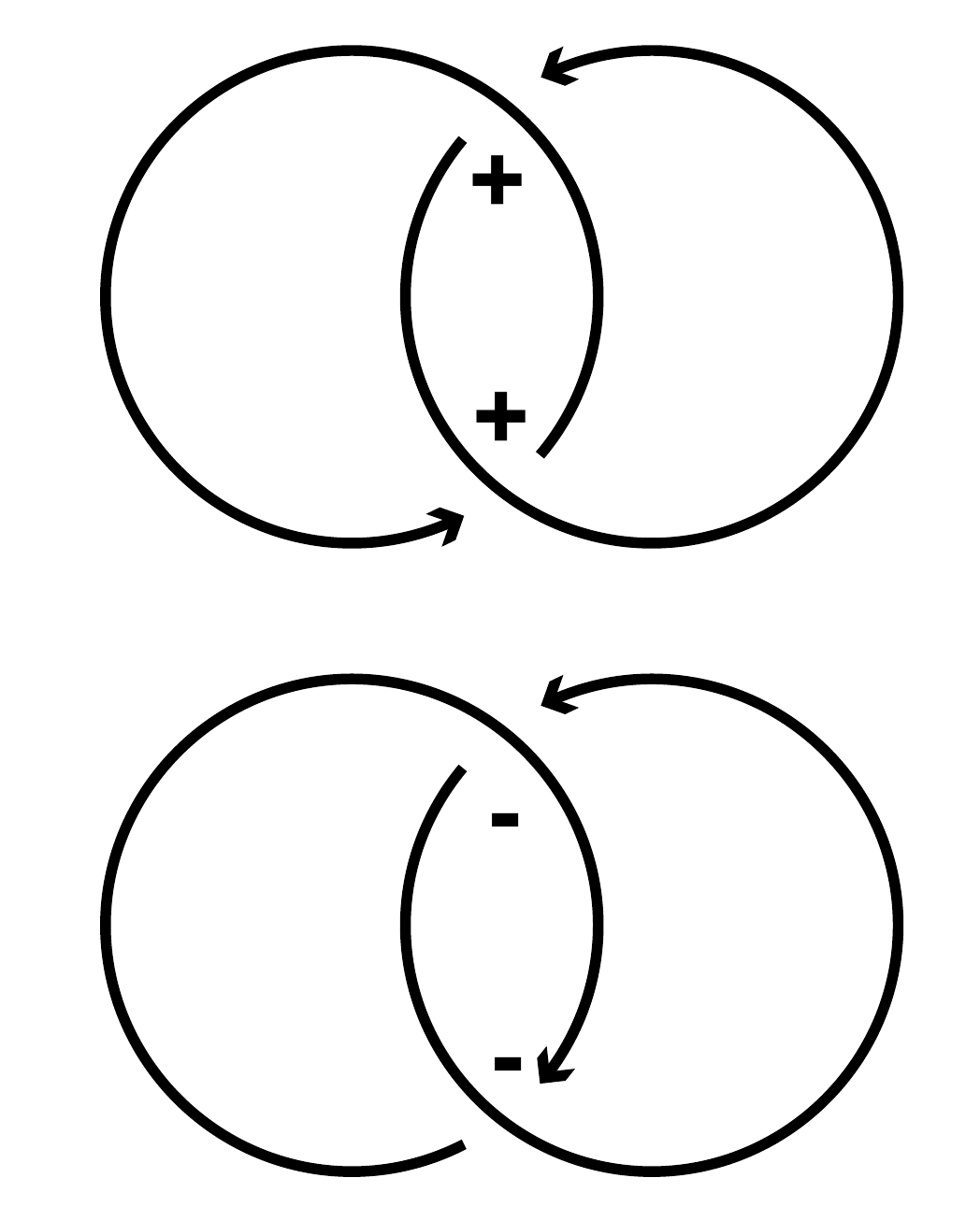}
  \caption{The sign markings on the trefoil and Hopf link with different orientations.}\label{figure:orientation}
\end{figure}
\end{example}

\begin{theorem}\label{theorem:sign}
Suppose two orientations of \(T\) have the same sign type. Then their Khovanov homology are isomorphic. Moreover, the corresponding generators have the same height and quantum grading.
\end{theorem}
\begin{proof}
Note that, regardless of the orientation, the underlying sequence is the same
\begin{equation*}
  \xymatrix{
  [[T]]^{0}\ar[r]^{d^{0}} & [[T]]^{1}\ar[r] & \cdots \ar[r] & [[T]]^{n-1}\ar[r]^{d^{n-1}} & [[T]]^{n}.
  }
\end{equation*}
Since the height shift is given by \(Kh^{k}(T) = \mathcal{G}([[T]]^{k+n_{-}})\)
and the orientations share the same sign type, it follows that their heights are equal.
Similarly, their quantum gradings coincide.
\end{proof}

In view of Theorem \ref{theorem:sign}, we can represent a tangle with a given orientation as \((T,\sigma)\), which uniquely determines its Khovanov homology. In this notation, we denote by \(H(T,\sigma)\) the Khovanov homology.

\begin{theorem}
Let $T$ be a tangle. Let $\sigma$ and $\tau$ be two sign types of $T$. Then we have an isomorphism
\begin{equation*}
   \Sigma: H^{k}(T,\sigma) \xrightarrow{\cong} H^{k+n_{-}(\tau)-n_{-}(\sigma)}(T,\tau),\quad k\in \mathbb{Z}.
\end{equation*}
Moreover, for any generator $[x]\in H^{k}(T,\sigma)$, we have
\begin{equation*}
  \Phi([x]) =  \Phi(\Sigma([x])) + 3n_{-}(\tau) - 3n_{-}(\sigma).
\end{equation*}
Here, $\Phi$ denote the quantum grading.
\end{theorem}
\begin{proof}
First, note that \((T,\sigma)\) and \((T,\tau)\) share the same underlying sequence:
\begin{equation*}
  \xymatrix{
  [[T]]^{0}\ar[r]^{d^{0}} & [[T]]^{1}\ar[r] & \cdots \ar[r] & [[T]]^{n-1}\ar[r]^{d^{n-1}} & [[T]]^{n}.
  }
\end{equation*}
By the height shift, we have
\begin{equation*}
  k_{\sigma} + n_{-}(\sigma) = k_{\tau} + n_{-}(\tau).
\end{equation*}
Hence, the map \(\Sigma\) is a height shift by \(n_{-}(\tau)-n_{-}(\sigma)\), sending \([x]\) to \([x]\) in the shifted degree.

From the formula for the quantum grading, we obtain
\begin{equation*}
  \Phi([x]) = \ell(x) + n - 3n_{-} + \theta(x),
\end{equation*}
where $n$ is the number of crossings and \(\ell(x)\) denotes the length of the state corresponding to \(x\).
It follows that
\begin{equation*}
  \Phi([x]) + 3n_{-}(\sigma) = \Phi(\Sigma([x])) + 3n_{-}(\tau),
\end{equation*}
which gives the desired result.
\end{proof}

Although differing by a global reversal of orientation, the Khovanov homology of a tangle under different orientations are isomorphic; however, the quantum gradings are shifted uniformly by a constant. This motivates us to normalize the quantum grading. For convenience, we may shift the quantum grading so that the generator with the lowest quantum degree is assigned degree zero. A more natural choice is to shift the quantum grading so that the generators of a certain unoriented invariant subspace have quantum degree zero. Here, the unoriented subspace refers to the part of the tangle complex that remains unaffected by the choice of orientation.

\subsubsection{Reidemeister move invariance}

Let us first consider the following result. The following proposition provides a computational proof verifying the invariance of Khovanov homology of  twisted tangles under $R1$ move.

\begin{proposition}
Let \(T\) be a $d$-twisted tangle, which refers to a tangle obtained by twisting a single arc in a given direction.
Then we have
\[
H^{k}(T) \cong \left\{
             \begin{array}{ll}
               \mathbb{K}\{x\}, & \hbox{$k=0$;} \\
               0, & \hbox{otherwise,}
             \end{array}
           \right.
\]
where $x$ is a generator of quantum grading $1$.
\begin{figure}[h]
  \centering
  \includegraphics[width=0.3\textwidth]{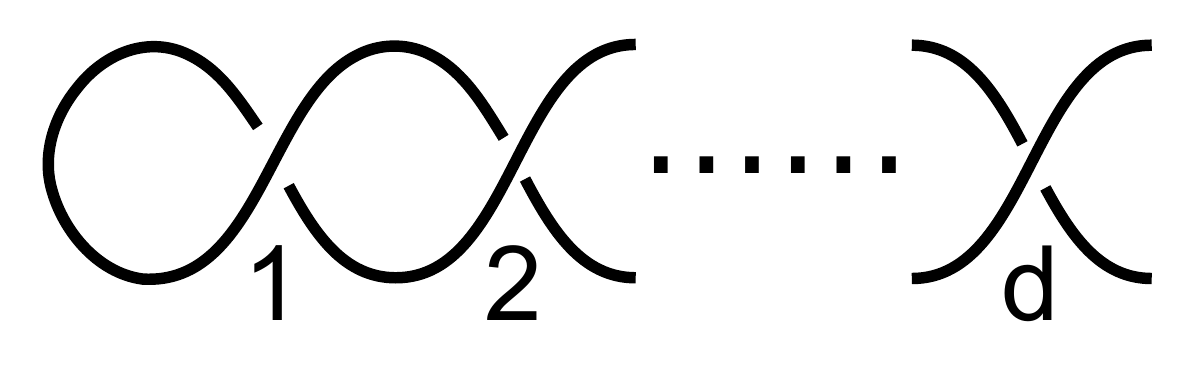}
  \caption{Illustration of a $d$-twisted tangle.}\label{figure:twist_tangle}
\end{figure}
\end{proposition}

\begin{proof}
As illustrated in Figure \ref{figure:twist_tangle}, we may always position the tangle so that its endpoints lie on the right. In this arrangement, the leftmost crossing can be either right-handed or left-handed. We prove the result by induction on $d$.
For $d = 0$, the tangle has no crossings, and its Khovanov homology is concentrated in degree $0$ with generator $[w]$ of quantum grading $-1$.
Assume that for $d = m-1$ we have
\[
H^{k}(T) \cong
\begin{cases}
\mathbb{K}\{x\}, & k = 0, \\
0, & \text{otherwise,}
\end{cases}
\]
where the generator $x$ has quantum grading $-1$.
We now prove the statement for $d = m$.

In the proof below, we restrict to the case where the leftmost crossing is right-handed; the case of a left-handed crossing can be treated analogously.

Let us denote by \(T^{+}(d)\) the tangle where the superscript \(+\) indicates that the leftmost crossing is right-handed, and \(d\) denotes the total number of crossings. Examining the smoothings of the leftmost crossing, we have
\begin{align*}
    & T^{+}(d)_{(0,s)} = \bigcirc \sqcup T^{-}(d-1)_{s}, \\
    & T^{+}(d)_{(1,s)} = T(d-1)_{s},
\end{align*}
where \(T^{-}(d-1)_{s}\) denotes the tangle with \(d-1\) crossings whose leftmost crossing is left-handed, and $\bigcirc$ denotes the single circle. Here, \(s\) is a state of \(T(d-1)\), while \((0,s)\) and \((1,s)\) are the states of \(T^{+}(d)\) obtained by taking the \(0\)-smoothing or \(1\)-smoothing at the leftmost crossing, respectively.

Thus, we obtain
\begin{equation*}
  \bigoplus\limits_{\ell(s')=k} T^{+}(d)_{s'} = \left(\bigoplus\limits_{\ell(s)=k} \bigcirc \sqcup T^{-}(d-1)_{s} \right)\oplus \left(\bigoplus\limits_{\ell(s)=k-1}  T^{-}(d-1)_{s} \right).
\end{equation*}
Recall that $Kh^{k}(T) = \mathcal{G}([[T]]^{k+n_{-}})$ and $\quad d_{T}^{k} = \mathcal{G}(d^{k+n_{-}})$. Thus, we have
\begin{equation*}
  Kh^{k}(T^{+}(d))  = \left( V\otimes Kh^{k}(T^{-}(d-1))\right) \oplus \left(Kh^{k-1}(T^{-}(d-1))\right).
\end{equation*}
Note that the differential on $Kh^{k-1}(T^{-}(d-1))$ is closed, that is,
\begin{equation*}
  d_{T^{-}(d-1)}\left(Kh^{k-1}(T^{-}(d-1))\right)\subseteq Kh^{k-1}(T^{-}(d-1)).
\end{equation*}
Hence, $Kh^{k-1}(T^{-}(d-1))$ is a sub cochain complex of $Kh^{k}(T^{+}(d))$. Thus, we have a short exact sequence
\begin{equation*}
  \xymatrix{
  0\ar[r] & Kh^{k-1}(T^{-}(d-1)) \ar[r]& Kh^{k}(T^{+}(d)) \ar[r]& (A,\bar{d})\ar[r] &0.
  }
\end{equation*}
where $A= Kh^{k}(T^{+}(d))/ Kh^{k-1}(T^{-}(d-1))$ is the quotient cochain complex with the induced differential $\bar{d}$. The short exact sequence induces a long exact sequence of homology
\begin{equation*}
  \xymatrix{
  \cdots\ar[r] & H^{-1}(T^{-}(d-1)) \ar[r]& H^{0}(T^{+}(d)) \ar[r]& H^{0}(A)\ar[r]^-{\delta} &  H^{0}(T^{-}(d-1))\ar[r]  &\cdots.
  }
\end{equation*}
By the induction hypothesis, we have $H^{k}(T^{-}(d-1))=0$ for $k\neq 0$. Note that
\begin{equation*}
  H^{0}(A) = V\otimes H^{0}(T^{-}(d-1)).
\end{equation*}
Hence, for $k\neq 0,1$, the long exact sequence yields the short exact sequence
\begin{equation*}
  \xymatrix{
    H^{k-1}(T^{-}(d-1))=0 \ar[r]& H^{k}(T^{+}(d)) \ar[r]& H^{k}(A)=0.
  }
\end{equation*}
which implies $H^{k}(T^{+}(d))=0$ for $k\neq 0,1$.

Let us examine the connecting homomorphism
\begin{equation*}
  \delta: H^{0}(A) \to H^{0}(T^{-}(d-1)),\quad [z]\mapsto [d_{T^{+}(d)}z].
\end{equation*}
By definition, any $z\in A$ can be written as either $v_{+}\otimes x$ or $v_{-}\otimes x$ for some $x\in Kh^{k-1}(T^{-}(d-1))$. Recall the maps
\begin{align*}
    & V\times V\to V,\quad  v_{+}\otimes v_{\pm} \mapsto v_{\pm},\\
    & V\times W\to W, \quad  v_{+}\otimes w \mapsto w.
\end{align*}
It follows that
\begin{equation*}
  d_{T^{+}(d)}(v_{+}\otimes x) = x + d_{T^{-}(d-1)}x.
\end{equation*}
If $[x]\in H^{0}(T^{-}(d-1))$, then $d_{T^{-}(d-1)}x=0$ and thus
\begin{equation*}
  d_{T^{+}(d)}(v_{+}\otimes x) = x.
\end{equation*}
Therefore, $[v_{+}\otimes x]$ is a generator of $H(A)$, and
\begin{equation*}
  \delta[v_{+}\otimes x] = [x].
\end{equation*}
This shows that $\delta: H^{0}(A) \to H^{0}(T^{-}(d-1))$ is surjective, sending $[v_{+}\otimes x]$ to $[x]$.

From the long exact sequence we then obtain
\[
\xymatrix{
0 \ar[r] & H^{1}(T^{+}(d)) \ar[r] & H^{1}(A) = 0
}
\]
which implies
\[
H^{1}(T^{+}(d)) = 0,
\]
and
\[
\xymatrix{
0 \ar[r] & H^{0}(T^{+}(d)) \ar[r] & H^{0}(A) \ar[r]^-{\delta} & H^{0}(T^{-}(d-1)) \ar[r] & 0
}
\]
which gives
\[
H^{0}(T^{+}(d)) \;\cong\; v_{-} \otimes H^{0}(T^{-}(d-1)) \;\cong\; H^{0}(T^{-}(d-1)).
\]
Consequently,
\[
H^{k}(T^{+}(d)) \;\cong\; H^{k}(T^{-}(d-1)).
\]

\paragraph{Quantum grading.}
Let $[x] \in H^{0}(T^{-}(d-1))$ be a generator of quantum grading $-1$. Then
\[
-1 = p(x) - n_{-} + n_{+} + \theta(x),
\]
where $p(x) = 0$ is the homological grading (or height), and $n_{-}$, $n_{+}$ are the numbers of left-handed and right-handed crossings in $T^{-}(d-1)$, respectively. This gives
\[
\theta(x) = n_{-} - n_{+} - 1.
\]
In $T^{+}(d)$, the number of left-handed crossings remains $n_{-}$, while the number of right-handed crossings becomes $n_{+} + 1$. Since
\[
\theta(v_{-} \otimes x) = \theta(x) - 1,
\]
the quantum grading of $[v_{-} \otimes x] \in H^{0}(T^{+}(d))$ is
\[
\Phi\big( [v_{-} \otimes x] \big)
= 0 - n_{-} + (n_{+} + 1) + (\theta(x) - 1) = -1.
\]

By induction, we conclude that
\[
H^{k}(T(d)) \cong
\begin{cases}
\mathbb{K}\{x\}, & k = 0, \\
0, & \text{otherwise},
\end{cases}
\]
where $x$ has quantum grading $-1$.
\end{proof}

\begin{theorem}\label{thm:tangle_invariance}
Khovanov homology of a tangle is invariant under the Reidemeister moves.
\end{theorem}
\begin{proof}
It is a direct result from \cite[Theorem~1]{bar2005khovanov} and \cite[Theorem~3.4]{liu2024persistent}.
\end{proof}

For Khovanov homology of tangles, the Reidemeister moves keep the quantum grading unchanged.

\subsubsection{Disjoint of tangles}

\begin{theorem}\label{thm:khovanov_tensor_product}
For two tangles $T$ and $T'$ with disjoint support, there exists a natural isomorphism of chain complexes
\[
 Kh(T \sqcup T') \cong Kh(T) \otimes Kh(T').
\]
Consequently, the Khovanov homology groups satisfy
\[
 H^{\ast}(T \sqcup T') \cong H^{\ast}(T) \otimes H^{\ast}(T'),
\]
where the tensor product is taken in the appropriate graded chain complex category.
\end{theorem}

\begin{proof}
Let $s \in \{0,1\}^{n(T)}$ be a state of $T$ in the cube of resolutions, and let $s' \in \{0,1\}^{n(T')}$ be a state of $T'$. Here, $n(T)$ and $n(T')$ denote the number of crossings of $T$ and $T'$, respectively. Then the corresponding state of $T \sqcup T'$ is $(s,s') \in \{0,1\}^{n(T)+n(T')}$.
It follows that
\[
(T \sqcup T')_{(s,s')} = T_s \sqcup T_{s'}.
\]
By definition, we have
\[
[[T \sqcup T']]^k = \bigoplus_{k=\ell(s,s')} (T \sqcup T')_{(s,s')} = \bigoplus_{k=\ell(s)+\ell(s')} T_s \sqcup T_{s'}.
\]
Noting that $\mathcal{G}(T_s \sqcup T_{s'}) = \mathcal{G}(T_s) \otimes \mathcal{G}(T_{s'})$, we obtain
\begin{align*}
  Kh^{k}(T+T') = & \mathcal{G}\left([[T+T']]^{k+n_{-}(T+T')}\right) \\
   = & \mathcal{G}\left(\bigoplus\limits_{k+n_{-}(T+T')=\ell(s)+\ell(s')} T_{s} \sqcup T_{s'}\right)\\
   = & \bigoplus\limits_{k+n_{-}(T)+n_{-}(T')=\ell(s)+\ell(s')} \mathcal{G}(T_{s}) \otimes \mathcal{G}(T_{s'}),
\end{align*}
where $n_-(T)$, $n_-(T')$, and $n_-(T \sqcup T')$ denote the numbers of left-handed crossings of $T$, $T'$, and $T \sqcup T'$, respectively.
Hence, we can write
\[
Kh^k(T \sqcup T') = \bigoplus_{k_1+k_2=k} \left(\bigoplus_{k_1+n_-(T)=\ell(s)} \mathcal{G}(T_s)\right) \otimes \left(\bigoplus_{k_2+n_-(T')=\ell(s')} \mathcal{G}(T_{s'})\right).
\]
This implies
\[
Kh^k(T \sqcup T') = \bigoplus_{k_1+k_2=k} Kh^{k_1}(T) \otimes Kh^{k_2}(T').
\]
Since the differential respects this decomposition, decomposing as the sum of the differentials on each factor tensored with the identity on the other, it follows that the chain complex for the disjoint union is isomorphic to the tensor product of the individual chain complexes, which in turn induces the claimed isomorphism on homology.
\end{proof}

\begin{corollary}
Let $T$ be a tangle with no crossings; that is, $T$ is a disjoint union of $r$ circles and $t$ arcs. Then the Khovanov homology of $T$ is given by
\[
H^{k}(T) \cong \left\{
             \begin{array}{ll}
               V^{\otimes r} \otimes W^{\otimes t}, & \hbox{$k=0$;} \\
               0, & \hbox{otherwise.}
             \end{array}
           \right.
\]
where $V=\mathbb{K}\{v_{+},v_{-}\}$ and $W=\mathbb{K}\{w\}$.
\end{corollary}

\subsection{Computation and examples}

\subsubsection{Computation procedure}

The procedure for computing the Khovanov homology of tangles is outlined as follows.

\begin{enumerate}[label=(\roman*)]
  \item \textbf{Orientation and crossing labeling}: \\
  For a given tangle, choose an arbitrary orientation. Label each crossing with an index, and mark it with ``$+$'' or ``$-$'' depending on whether it is right-handed or left-handed. Obtain the Gauss code, and record the numbers $n_{+}$ and $n_{-}$, with $n = n_{+} + n_{-}$.

  \item \textbf{Diagram representation}: \\
  For each state $s$, draw the corresponding smoothed diagram $T_s$. Classify these diagrams according to the length $\ell(s)$, resulting in $n$ classes; the $k$-th class contains $\binom{n}{k}$ tangles. Identify the edges $d_{\xi}$ in the $n$-cube connecting states from class $k$ to class $k+1$, and link the corresponding diagrams $T_s$ and $T_{s'}$. Repeat this for all such connections.

  \item \textbf{Applying construction $G$}: \\
  Associate each diagram with a vector space by construction $\mathcal{G}$, and assign to each edge $d_\xi$ a linear map $\mathcal{G}(d_{\xi})$.

  \item \textbf{Constructing the cochain complex}: \\
  For all states with $\ell(s) = k$, take the direct sum
  \[
    \bigoplus_{\ell(s) = k} \mathcal{G}(T_s),
  \]
  and assign to this term the height $k - n_{-}$. The differential on this space is given by
  \[
    d^{k} = \sum_{|\xi| = k} (-1)^{\mathrm{sgn}(\xi)} \, \mathcal{G}(d_{\xi}).
  \]

  \item \textbf{Computing homology}: \\
  Compute the homology of the cochain complex by
  \[
    H^{k}(T) = \frac{\ker d^{k}}{\mathrm{im}\, d^{k-1}}.
  \]
  For each generator, determine its quantum grading and normalize the gradings accordingly.
\end{enumerate}

\subsubsection{Illustrative examples}

\begin{example}
Consider the crossing tangle shown in Figure~\ref{figure:crossing_tangle}. There are two possible orientations of the crossing tangle: right-handed and left-handed. We compute the corresponding Khovanov homology in each case.
\begin{figure}[h]
  \centering
  \includegraphics[width=0.25\textwidth]{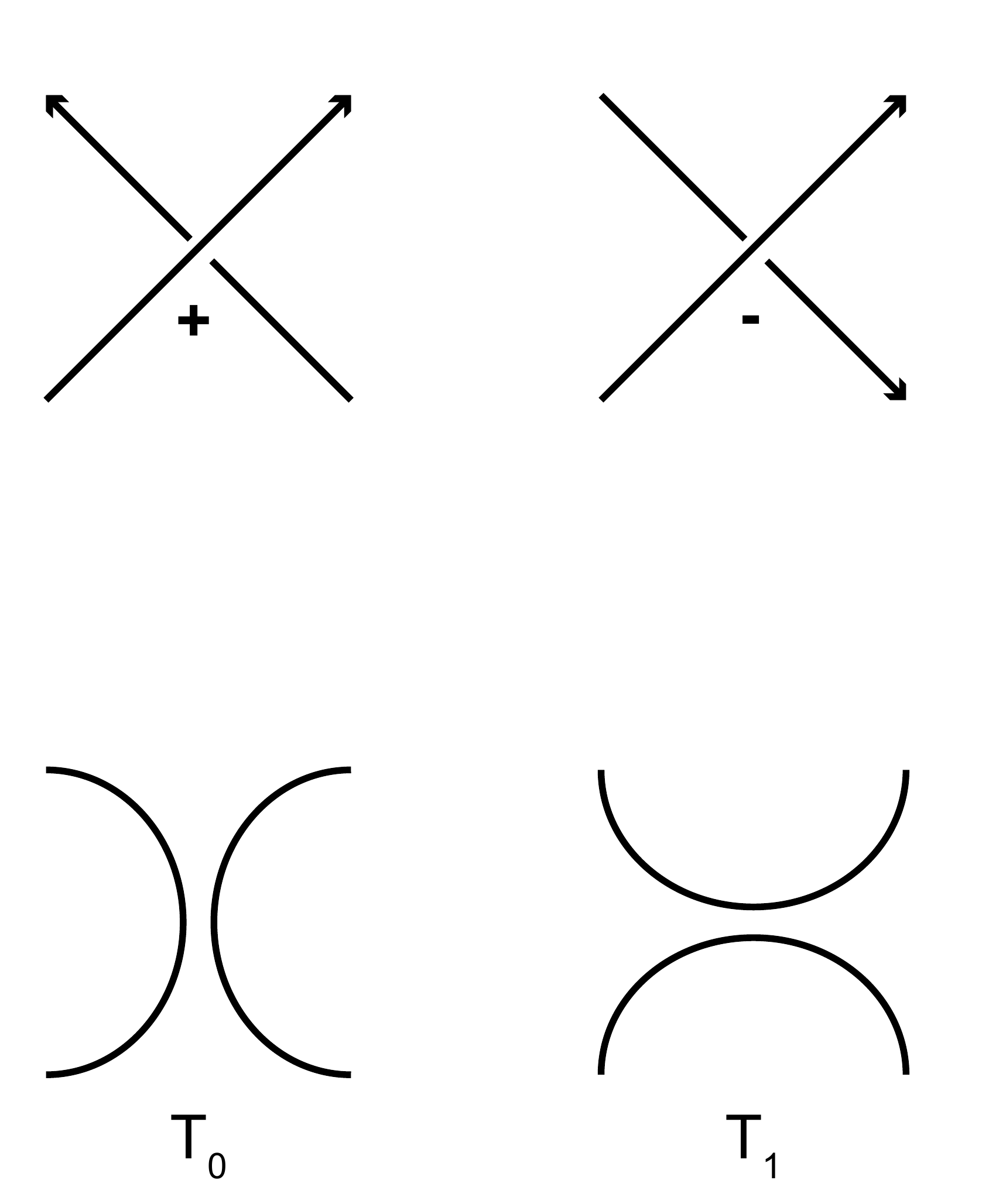}
  \caption{Crossing tangle with different orientations, and the cube of smoothing resolutions for the crossing tangle.}\label{figure:crossing_tangle}
\end{figure}

We first consider the right-handed crossing tangle $T$.

(i) For the right-handed crossing tangle $T$, we have $n_{+} = 1$ and $n_{-} = 0$.

(ii) The cube of resolutions has only two vertices, corresponding to the states $(0)$ and $(1)$. The associated tangle diagrams are shown in Figure~\ref{figure:crossing_tangle}.
The map $d_{(\star)} \colon T_{(0)} \to T_{(1)}$ corresponds to the morphism given by $\raisebox{-0.15cm}{\saddle{0.3}}$.

(iii) Applying the construction $\mathcal{G}$, we obtain
\begin{eqnarray*}
  \mathcal{G}(T_{0}) &=& W\otimes W, \\
  \mathcal{G}(T_{1}) &=& W\otimes W,
\end{eqnarray*}
and
\[
  \mathcal{G}(d_{(\star)}) = \mathcal{G}(T_{(0)} \to T_{(1)}),
  \quad w\otimes w \mapsto 0.
\]

(iv) We then obtain the cochain complex
\[
  \xymatrix{
    0 \ar[r] & W\otimes W \ar[r]^-{d^{0}} & W\otimes W \ar[r] & 0,
  }
\]
where the height of $\mathcal{G}(T_{0}) = W\otimes W$ is $\ell((0))-n_{-}=0$, and the height of $\mathcal{G}(T_{1})=W\otimes W$ is $1$.
The differential is given by
\[
  d^{0}(w\otimes w) = 0.
\]

(v) A direct calculation shows that
\[
\begin{split}
    & \ker d^{0} = W\otimes W, \quad \mathrm{im}\, d^{0} = 0, \\
    & \ker d^{1} = W\otimes W, \quad \mathrm{im}\, d^{1} = 0.
\end{split}
\]
It follows that
\[
  H^{k}(T,(+)) =
  \begin{cases}
    W\otimes W, & k = 0,1, \\
    0, & \text{otherwise}.
  \end{cases}
\]

Similarly, the Khovanov homology of the left-handed crossing tangle can be computed.
The only difference is that $n_{-} = 1$ and $n_{+} = 0$, which yields
\[
  H^{k}(T,(-)) =
  \begin{cases}
    W\otimes W, & k = -1,0, \\
    0, & \text{otherwise}.
  \end{cases}
\]
It is noteworthy that choosing different orientations generally leads to different computed Khovanov homology for the tangle.

\end{example}

\begin{example}\label{example:twist}
Consider the following tangle shown in Figure \ref{figure:arc_tangle}. This tangle is obtained from a single arc with a self-crossing. We compute its Khovanov homology as follows.
\begin{figure}[h]
  \centering
  \includegraphics[width=0.2\textwidth]{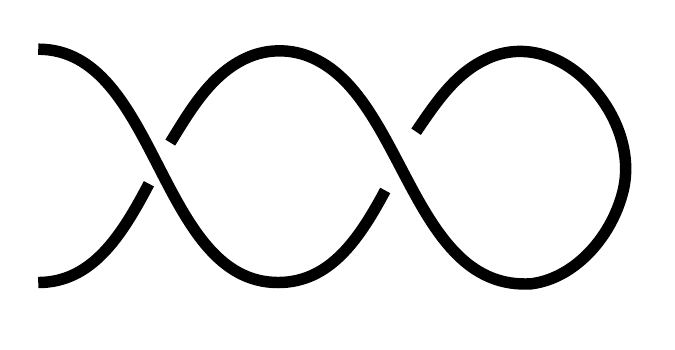}
  \includegraphics[width=0.2\textwidth]{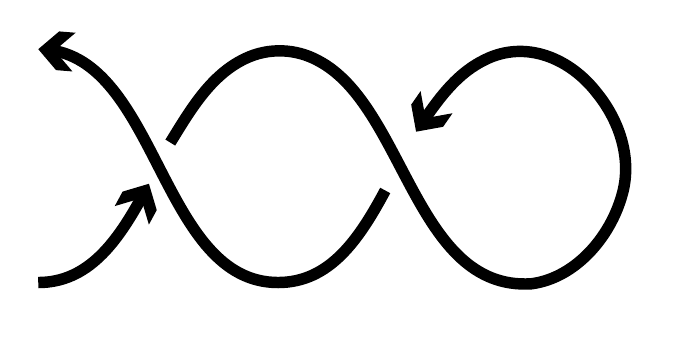}
  \caption{The $2$-twisted tangle and its orientation.}\label{figure:arc_tangle}
\end{figure}

(i) According to the orientation in Figure \ref{figure:arc_tangle}, the tangle contains two left-handed crossings. Therefore, $n=n_{-}=2$ and $n_{+}=0$.

(ii) We consider all states $(0,0)$, $(0,1)$, $(1,0)$, $(1,1)$, and their corresponding tangle diagrams are shown in Figure \ref{figure:arc_cube}.
\begin{figure}[h]
  \centering
  \includegraphics[width=0.75\textwidth]{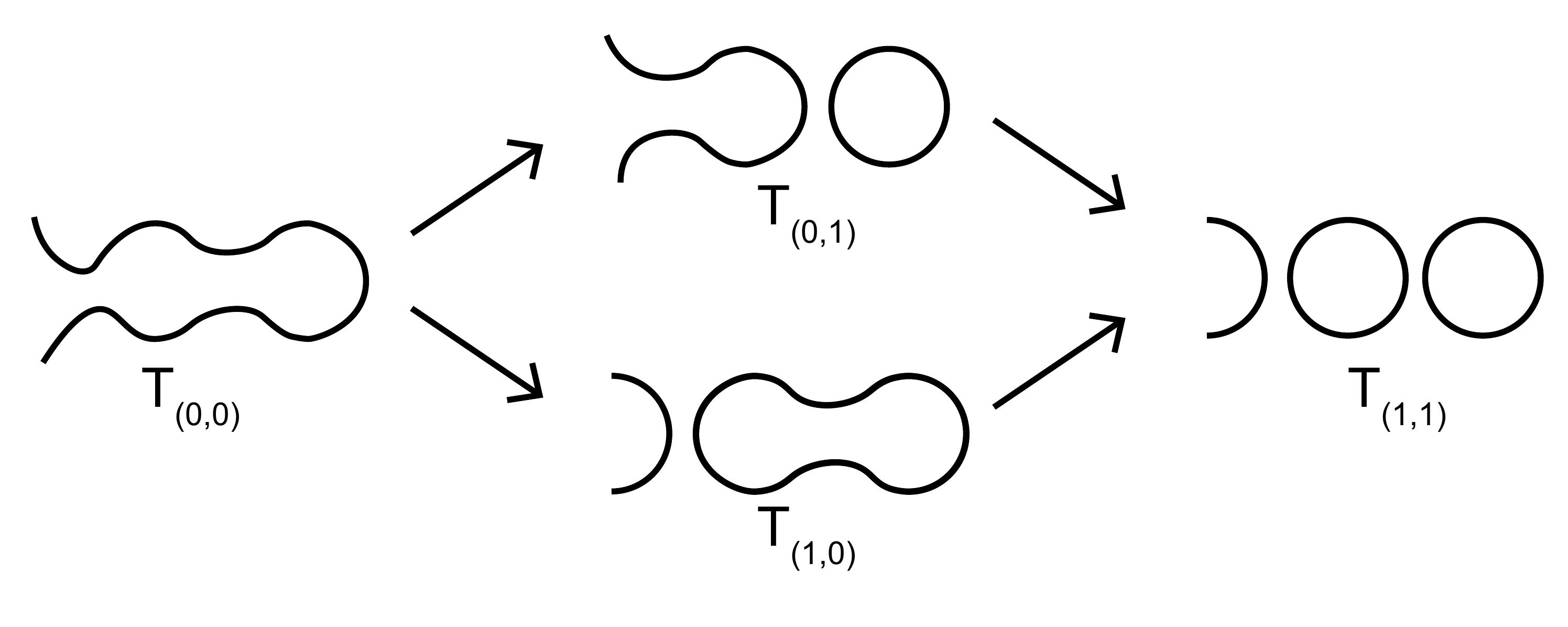}
  \caption{Cube of smoothing resolutions for the twisted tangle in Example \ref{example:twist}.}\label{figure:arc_cube}
\end{figure}
It can be observed that $d_{(0,\star)}:T_{(0,0)}\to T_{(0,1)}$ corresponds to the morphism given by $\raisebox{-0.15cm}{\saddle{0.3}}$ at the right crossing, while $d_{(\star,0)}: T_{(0,0)}\to T_{(1,0)}$ corresponds to $\raisebox{-0.15cm}{\saddle{0.3}}$ at the left crossing. Similarly, $d_{(\star,1)}: T_{(0,1)}\to T_{(1,1)}$ and $d_{(1,\star)}:T_{(1,0)}\to T_{(1,1)}$ are given by $\raisebox{-0.15cm}{\saddle{0.3}}$ at the left crossing.

(iii) Using the functor $\mathcal{G}$, we obtain
\begin{eqnarray*}
  \mathcal{G}(T_{(0,0)}) &=& W, \\
  \mathcal{G}(T_{(0,1)}) &=& W\otimes V,\\
  \mathcal{G}(T_{(1,0)}) &=& W\otimes V ,\\
  \mathcal{G}(T_{(1,1)}) &=& W\otimes V \otimes V .
\end{eqnarray*}
Accordingly, the linear maps are
\begin{align*}
  \mathcal{G}(d_{(0,\star)}) = \mathcal{G}(T_{(0,0)}\to T_{(0,1)}), \quad & w\mapsto w\otimes v_{-}, \\
  \mathcal{G}(d_{(\star,0)}) = \mathcal{G}(T_{(0,0)}\to T_{(1,0)}), \quad & w\mapsto w\otimes v_{-}, \\
  \mathcal{G}(d_{(\star,1)}) = \mathcal{G}(T_{(0,1)}\to T_{(1,1)}), \quad &
    \left\{
      \begin{array}{ll}
        w\otimes v_{-}\mapsto (w\otimes v_{-})\otimes v_{-},  \\
        w\otimes v_{+}\mapsto (w\otimes v_{+})\otimes v_{-},
      \end{array}
    \right. \\
  \mathcal{G}(d_{(1,\star)}) = \mathcal{G}(T_{(1,0)}\to T_{(1,1)}), \quad &
    \left\{
      \begin{array}{ll}
        w\otimes v_{-}\mapsto -\, w\otimes (v_{-}\otimes v_{-}),  \\
        w\otimes v_{+}\mapsto -\, w\otimes (v_{+}\otimes v_{-}+ v_{-}\otimes v_{+}).
      \end{array}
    \right.
\end{align*}

(iv) The resulting cochain complex is
\begin{equation*}
  \xymatrix{
    0\ar@{->}[r] & W \ar@{->}[r]^-{d^{-2}} & (W\otimes V)\oplus (W\otimes V) \ar@{->}[r]^-{d^{-1}} & W\otimes V\otimes V \ar@{->}[r] & 0.
  }
\end{equation*}
The heights of $W$, $(W\otimes V)\oplus (W\otimes V)$, and $W\otimes V\otimes V$ are $-2$, $-1$, and $0$, respectively. The differentials are explicitly
\begin{equation*}
  d^{-2}(w) =  \left(
                \begin{array}{cccc}
                  0 & 1 & 0 & 1
                \end{array}
              \right)
   \left(
                 \begin{array}{c}
                   (w\otimes v_{+},0) \\
                   (w\otimes v_{-},0) \\
                   (0,w\otimes v_{+}) \\
                   (0,w\otimes v_{-})
                 \end{array}
               \right)
\end{equation*}
and
\begin{equation*}
  d^{-1}\left(
                 \begin{array}{c}
                   (w\otimes v_{+},0) \\
                   (w\otimes v_{-},0) \\
                   (0,w\otimes v_{+}) \\
                   (0,w\otimes v_{-})
                 \end{array}
               \right) =\left(
                          \begin{array}{cccc}
                            0 & 1 & 0 & 0 \\
                            0 & 0 & 0 & 1 \\
                            0 & -1 & -1 & 0 \\
                            0 & 0 & 0 & -1
                          \end{array}
                        \right)
                \left(
                           \begin{array}{c}
                             w\otimes v_{+} \otimes v_{+} \\
                             w\otimes v_{+} \otimes v_{-} \\
                             w\otimes v_{-} \otimes v_{+} \\
                             w\otimes v_{-} \otimes v_{-}
                           \end{array}
                         \right).
\end{equation*}

(v) By a straightforward calculation, we have
\begin{eqnarray*}
  \ker d^{-2} &=& 0, \\
  \im d^{-2} &=& \langle (w\otimes v_{-},w\otimes v_{-})\rangle, \\
  \ker d^{-1} &=& \langle (w\otimes v_{-},w\otimes v_{-})\rangle,\\
  \im d^{-1} &=& \langle w\otimes v_{+} \otimes v_{-}, w \otimes v_{-} \otimes v_{+}, w\otimes v_{-} \otimes v_{-}\rangle.
\end{eqnarray*}
Thus, the Khovanov homology is
\begin{equation*}
  H^{k}(T) = \left\{
            \begin{array}{ll}
              [w\otimes v_{+} \otimes v_{+}], & k =0,\\
              0, & \text{otherwise}.
            \end{array}
          \right.
\end{equation*}
The quantum grading of the generator $[w\otimes v_{+} \otimes v_{+}]$ is
\begin{equation*}
  \Phi([w\otimes v_{+} \otimes v_{+}]) = 0 + n_{+} - n_{-} + \theta(w\otimes v_{+} \otimes v_{+}) = 0 + 0 -2+1=-1.
\end{equation*}
This result coincides with the Khovanov homology of a single crossing-free arc, as the given tangle can be transformed into a crossing-free arc by two $R1$-moves.
\end{example}

\begin{example}\label{example:tritangle}
Consider the tangle depicted in Figure \ref{figure:tritangle}. We proceed to compute its Khovanov homology.

\begin{figure}[h]
  \centering
  \includegraphics[width=0.15\textwidth]{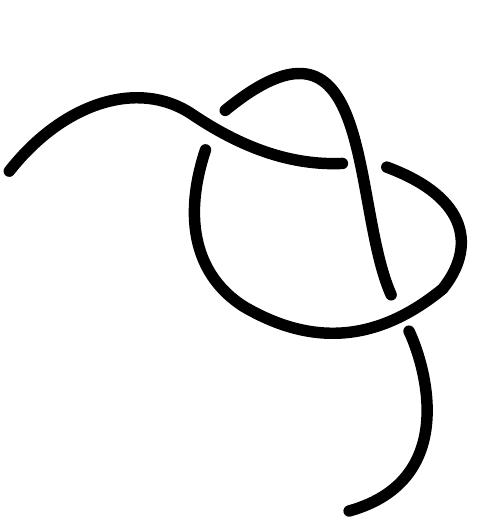}\qquad \qquad
  \includegraphics[width=0.15\textwidth]{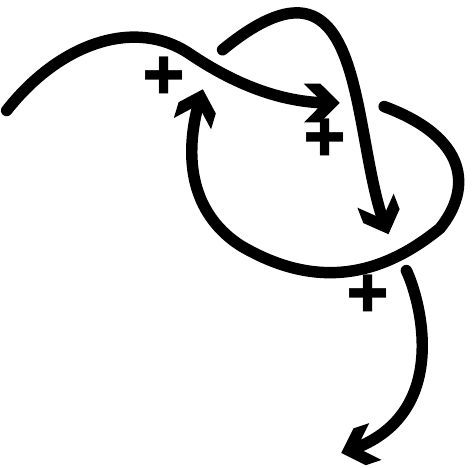}
  \caption{The tangle in Example \ref{example:tritangle} and its orientation.}\label{figure:tritangle}
\end{figure}

(i) First, we fix the orientation as shown in Figure \ref{figure:tritangle}. It can be observed that $n_{+}=3$ and $n_{-}=0$.

(ii) We now obtain eight states, given by
\begin{align*}
    & (0,0,0), (1,0,0), (0,1,0), (0,0,1), \\
    & (1,1,0), (1,0,1), (0,1,1), (1,1,1).
\end{align*}
The corresponding tangle diagrams are shown in Figure \ref{figure:tritangle_cube}.
\begin{figure}[h]
  \centering
  \includegraphics[width=0.6\textwidth]{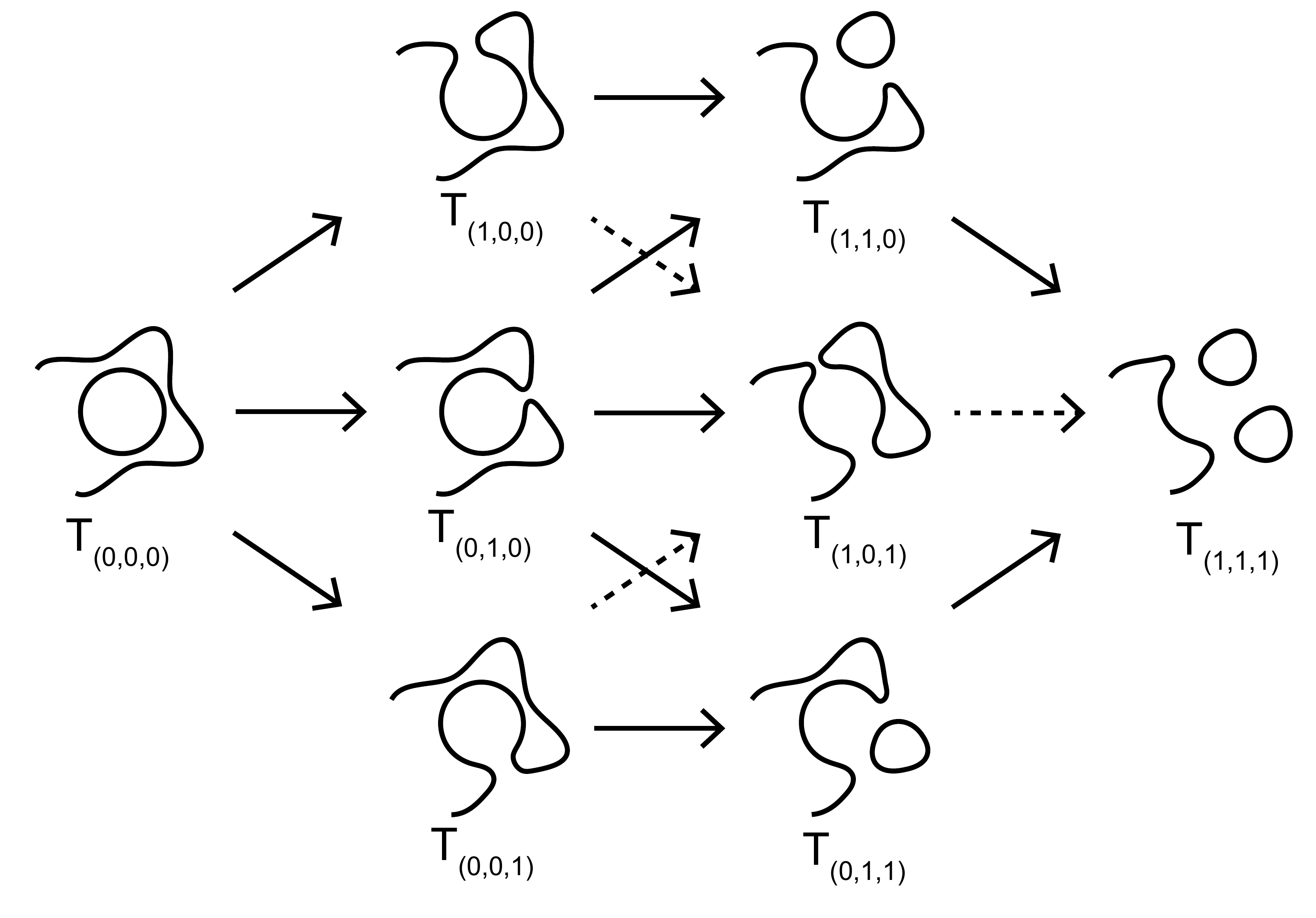}
  \caption{Cube of smoothing resolutions for the tangle in Example \ref{example:tritangle}.}\label{figure:tritangle_cube}
\end{figure}
We can then describe the connecting map as follows.
\begin{align*}
  & d_{(\star,0,0)}=\raisebox{-0.15cm}{\saddle{0.3}}:T_{(0,0,0)}\to T_{(1,0,0)},\quad d_{(0,\star,0)}=\raisebox{-0.15cm}{\hsaddle{0.3}}:T_{(0,0,0)}\to T_{(0,1,0)}, \\
  & d_{(0,0, \star)}=\raisebox{-0.15cm}{\saddle{0.3}}:T_{(0,0,0)}\to T_{(0,1,1)},\quad d_{(1,\star,0)}=\raisebox{-0.15cm}{\hsaddle{0.3}}:T_{(1,0,0)}\to T_{(1,1,0)}, \\
  & d_{(1,0,\star)}=\raisebox{-0.15cm}{\saddle{0.3}}:T_{(1,0,0)}\to T_{(1,0,1)},\quad d_{(\star,1,0)}=\raisebox{-0.15cm}{\saddle{0.3}}:T_{(0,1,0)}\to T_{(1,1,0)},\\
  & d_{(0,1,\star)}=\raisebox{-0.15cm}{\hsaddle{0.3}}:T_{(0,1,0)}\to T_{(0,1,1)},\quad d_{(\star,0,1)}=\raisebox{-0.15cm}{\saddle{0.3}}:T_{(0,0,1)}\to T_{(1,0,1)},\\
  & d_{(0, \star,1)}=\raisebox{-0.15cm}{\hsaddle{0.3}}:T_{(0,0,1)}\to T_{(0,1,1)},\quad d_{(1,1,\star)}=\raisebox{-0.15cm}{\saddle{0.3}}:T_{(1,1,0)}\to T_{(1,1,1)},\\
  & d_{(1,\star,1)}=\raisebox{-0.15cm}{\hsaddle{0.3}}:T_{(1,0,1)}\to T_{(1,1,1)},\quad d_{(\star,1,1)}=\raisebox{-0.15cm}{\saddle{0.3}}:T_{(0,1,1)}\to T_{(1,1,1)}.\\
\end{align*}

(iii) According to the construction of $\mathcal{G}$ on tangles, we have
\begin{align*}
    & \mathcal{G}(T_{(0,0,0)}) = W\otimes V, \\
    & \mathcal{G}(T_{(1,0,0)}) = \mathcal{G}(T_{(0,1,0)})= \mathcal{G}(T_{(0,0,1)}) = W,\\
    & \mathcal{G}(T_{(1,1,0)}) = \mathcal{G}(T_{(1,0,1)}) = \mathcal{G}(T_{(0,1,1)}) = W\otimes V,\\
    & \mathcal{G}(T_{(1,1,1)}) = W\otimes V\otimes V.
\end{align*}
Moreover, we have $\mathcal{G}( d_{(\star,0,0)}) = \mathcal{G}( d_{(0,\star,0)}) =  \mathcal{G}(d_{(0,0,\star)})$ on their respective components, which are given by
\begin{equation*}
   W\otimes  V\to W,\quad \left\{
                            \begin{array}{ll}
                                 w\otimes v_{+}\mapsto w,\\
                                 w\otimes v_{-}\mapsto 0,
                            \end{array}
                          \right.
\end{equation*}
The maps
\begin{equation*}
  \mathcal{G}( d_{(1,\star,0)}) = \mathcal{G}( d_{(1, 0,\star)}) =  \mathcal{G}(d_{(\star,1,0)})=\mathcal{G}(d_{(0,1,\star)})= \mathcal{G}(d_{(\star,0,1)})=\mathcal{G}(d_{(0,\star,1)})
\end{equation*}
on their respective components are given by
\begin{equation*}
  W\to W\otimes  V,\quad w\mapsto w\otimes v_{-}.
\end{equation*}
The maps $\mathcal{G}( d_{(1, 1,\star)}) =  \mathcal{G}(d_{(\star,1,1})$ on their respective components are represented by
\begin{equation*}
  W\otimes  V\to W\otimes  V\otimes  V,\quad \left\{
                                                        \begin{array}{ll}
                                                          w\otimes v_{+} \mapsto w\otimes v_{-} \otimes v_{+}, \\
                                                          w\otimes v_{-} \mapsto w\otimes v_{-} \otimes v_{-}
                                                        \end{array}
                                                      \right.
\end{equation*}
and the map $\mathcal{G}( d_{(1,\star, 1)})$ is described by
\begin{equation*}
  W\otimes  V\to W\otimes  V\otimes  V,\quad \left\{
                                                        \begin{array}{ll}
                                                          w\otimes v_{+} \mapsto w\otimes v_{+} \otimes v_{-} + w\otimes v_{-} \otimes v_{+}, \\
                                                          w\otimes v_{-} \mapsto w\otimes v_{-} \otimes v_{-}.
                                                        \end{array}
                                                      \right.
\end{equation*}

(iv) Altogether, we have the following cochain complex
\begin{small}
\begin{equation*}
  \xymatrix@C=0.6cm{
    0 \ar[r] & W\otimes V \ar[r]^-{d^{0}}& W\oplus W\oplus W \ar[r]^-{d^{1}} & (W\otimes V)\oplus (W\otimes V)\oplus(W\otimes V)\ar[r]^-{d^{2}}  & W\otimes V\otimes V\ar[r]&  0.
  }
\end{equation*}
\end{small}
The heights of the corresponding spaces, from left to right, are $0$, $1$, and $2$, respectively. Moreover, the differentials are given by
\begin{equation*}
  d^{0}\left(
         \begin{array}{c}
          w \otimes v_{+} \\
          w \otimes v_{-} \\
         \end{array}
       \right) = \left(
                   \begin{array}{ccc}
                     1 & 1 & 1 \\
                     0 & 0 & 0 \\
                   \end{array}
                 \right)
        \left(
                    \begin{array}{c}
                      (w,0,0) \\
                      (0,w,0) \\
                      (0,0,w) \\
                    \end{array}
                  \right),
\end{equation*}
\begin{equation*}
  d^{1}\left(
                    \begin{array}{c}
                      (w,0,0) \\
                      (0,w,0) \\
                      (0,0,w) \\
                    \end{array}
                  \right) = \left(
                              \begin{array}{cccccc}
                                0 & -1 & 0 & -1 & 0 & 0 \\
                                0 & 1 & 0 & 0 & 0 & -1 \\
                                0 & 0 & 0 & 1 & 0 & 1 \\
                              \end{array}
                            \right)
                   \left(
                              \begin{array}{c}
                                (w \otimes v_{+},0,0) \\
                                (w \otimes v_{-},0,0) \\
                                (0,w \otimes v_{+},0) \\
                                (0,w \otimes v_{-},0) \\
                                (0,0, w \otimes v_{+}) \\
                                (0,0, w \otimes v_{-}) \\
                              \end{array}
                            \right),
\end{equation*}
and
\begin{equation*}
  d^{2}\left(
                              \begin{array}{c}
                                (w \otimes v_{+},0,0) \\
                                (w \otimes v_{-},0,0) \\
                                (0,w \otimes v_{+},0) \\
                                (0,w \otimes v_{-},0) \\
                                (0,0, w \otimes v_{+}) \\
                                (0,0, w \otimes v_{-}) \\
                              \end{array},
                            \right) = \left(
                                        \begin{array}{cccc}
                                          0 & 1 & 0 & 0 \\
                                          0 & 0 & 0 & 1 \\
                                          0 & -1 & -1 & 0 \\
                                          0 & 0 & 0 & -1 \\
                                          0 & 0 & 1 & 0 \\
                                          0 & 0 & 0 & 1 \\
                                        \end{array}
                                      \right)
                             \left(
                                         \begin{array}{c}
                                           w\otimes v{+}\otimes v{+} \\
                                           w\otimes v{+}\otimes v{-} \\
                                           w\otimes v{-}\otimes v{+} \\
                                           w\otimes v{-}\otimes v{-} \\
                                         \end{array}
                                       \right).
\end{equation*}
As a result, we obtain
\begin{align*}
    & \ker d^{0} = \langle w\otimes v_{-}\rangle ,\quad \im d^{0} =\langle (w,w,w)\rangle,\quad \ker d^{1} = \langle (w,w,w)\rangle ,\quad\\
    &  \im d^{1} =\langle (w \otimes v_{-},w \otimes v_{-},0),(w \otimes v_{-},0, -w \otimes v_{-})\rangle, \\
    & \ker d^{2} =\langle (w \otimes v_{-},w \otimes v_{-},0),(w \otimes v_{-},0, -w \otimes v_{-}),(w \otimes v_{+},w \otimes v_{+}, w \otimes v_{+}) \rangle,\\
    & \im d^{2}  = \langle w\otimes v{+}\otimes v{-},w\otimes v{-}\otimes v{+} ,w\otimes v{-}\otimes v{-}\rangle.
\end{align*}
Hence, the Khovanov homology of $T$ is
\begin{equation*}
  H^{k}(T) = \left\{
               \begin{array}{ll}
                 \mathbb{K}\{[w\otimes v_{-}]\}, & \hbox{$k=0$;} \\
                 \mathbb{K}\{[(w \otimes v_{+},w \otimes v_{+}, w \otimes v_{+}) ]\}, & \hbox{$k=2$;} \\
                 \mathbb{K}\{[(w \otimes v_{+}\otimes v_{+})]\}, & \hbox{$k=3$;} \\
                 0, & \hbox{otherwise.}
               \end{array}
             \right.
\end{equation*}

\end{example}

\section{Algorithm and examples}

In this section, we revisit the computation of Khovanov homology for tangles from the perspective of code implementation. We provide a step-by-step overview of how smoothing states, differentials, and gradings can be represented and manipulated algorithmically. In addition, we present several usage examples to illustrate practical computations and applications of the Khovanov complex.

\subsection{Algorithmic treatment of smoothing states}

Let $D$ be a planar diagram code with $n$ crossings. A smoothing state is a binary sequence $s = (s_1, s_2, \dots, s_n) \in \{0,1\}^n$, where each $s_i$ determines whether a 0-smoothing or 1-smoothing is applied at crossing $i$.

Each smoothing state produces a set of connected components in the smoothed diagram, as described in the previous section. We denote the number of \texttt{circles} and \texttt{arcs} in the smoothed diagram associated with state $s$ as $c(s)$ and $a(s)$.

For each homological degree $k$, the Khovanov chain group $C^k = \mathcal{G}([[T]]^{k})$ is a direct sum over all states $s$ with $|s| = k$:
\[
C^k = \bigoplus_{\substack{s \in \{0,1\}^n \\ |s| = k}} V^{\otimes c(s)}\otimes W^{\otimes a(s)},
\]
where $V$ is a graded vector space generated by basis elements $v_+$ and $v_-$ with quantum gradings $+1$ and $-1$, respectively. And $W$ is a graded vector space generated by basis element $w$ with quantum grading $-1$.

Each tensor factor in $V^{\otimes c(s)}\otimes W^{\otimes a(s)}$ corresponds to a \texttt{circle} or \texttt{arc} in the smoothed diagram associated with a given smoothing state $s$. Therefore, from a computational standpoint, it is essential to develop algorithms that can efficiently identify and classify these components, distinguishing between circles and arcs, within the smoothed diagram. These algorithms form the foundation for constructing the chain groups in the Khovanov complex.

Given a planar diagram code consisting of a list of crossings
\[
P = [p_1, p_2, \dots, p_n],
\]
where each crossing $p_i$ is a 4-tuple of strand labels
\[
p_i = [a_i, b_i, c_i, d_i],
\]
we consider the smoothing state $s = (s_1, s_2, \dots, s_n)$, with $s_i \in \{0,1\}$, which specifies how to resolve each crossing.

For each crossing $p_i$:
\begin{itemize}
  \item If $s_i = 0$, we connect $a_i \sim d_i$ and $b_i \sim c_i$ (called 0-smoothing or A-smoothing);
  \item If $s_i = 1$, we connect $a_i \sim b_i$ and $c_i \sim d_i$ (called 1-smoothing or B-smoothing).
\end{itemize}

These local connections collectively define a binary relation $\sim$ on the set of strand labels. Specifically, for each smoothing at crossing $p_i$, we add two relations depending on the smoothing type. By repeatedly identifying all labels connected through sequences of such pairwise relations, we obtain a partition of the strand label set into disjoint classes, each corresponding to a connected component in the smoothed diagram.

\[
C_1,\quad C_2,\quad \dots,\quad C_k,
\]
where each $C_j$ is a set of mutually connected strand labels
\[
C_j = \{ \alpha_1,\alpha_2, \dots, \alpha_m \}.
\]

These classes represent the connected components in the smoothed diagram. To determine whether a component is a \texttt{circle} or an \texttt{arc}, we examine the labels of strands in each class. We classify the component as
\begin{itemize}
  \item \texttt{arc}, if any strand in the class contains a ``\texttt{|}'';
  \item \texttt{circle}, if all strands in the class are purely internal.
\end{itemize}

The final output of the smoothing algorithm is a list of connected components, each represented as an object with:
\begin{itemize}
  \item a \texttt{Type}, which is either \texttt{arc} or \texttt{circle};
  \item a \texttt{Representative}, which is the list of strand labels in the component.
\end{itemize}

\subsection{Usage example: smoothing state computation} \label{section:example}

The following example demonstrates how to use the \texttt{SmoothingStateGenerator} to compute the smoothed components from a given planar diagram code.

\paragraph{Step 1: Define the PD code.}

\begin{verbatim}
pdcode = [
    ["3", "10|", "4", "9"],
    ["|1", "8", "2", "|7"],
    ["9", "4", "8", "5"],
    ["2", "5", "3", "6|"]
]
\end{verbatim}

\paragraph{Step 2: Initialize the generator and query smoothing states.}

\begin{verbatim}
smoothing_state_generator = SmoothingStateGenerator(pdcode)
print(smoothing_state_generator.get_smoothing_state("0000"))
print(smoothing_state_generator.get_smoothing_state("0001"))
\end{verbatim}

\paragraph{Output.}

\begin{verbatim}
Smoothing State: 0000
10|~3~6| (Type: arc)
4~9 (Type: circle)
2~5~8~|1~|7 (Type: arc)

Smoothing State: 0001
10|~3~5~8~|1 (Type: arc)
4~9 (Type: circle)
2~6|~|7 (Type: arc)
\end{verbatim}

\subsection{Local maps and the differential}

To understand the differential in the Khovanov complex, we consider pairs of smoothing states $s, s' \in \{0,1\}^n$ such that $|s| = i$, $|s'| = i+1$, and $s'$ differs from $s$ at exactly one index $j$, where $s_j = 0$ and $s_j' = 1$.

Each such local modification induces a cobordism between the smoothed diagrams associated to $s$ and $s'$, which in turn leads to a linear map
\[
d_{s \to s'} : W^{\otimes a(s)} \otimes V^{\otimes c(s)} \longrightarrow W^{\otimes a(s')} \otimes V^{\otimes c(s')}.
\]

Given two smoothing states $s \to s'$ that differ at a single crossing, the local change at that crossing alters the connected components in the diagram. By comparing the smoothed diagrams before and after the change, we determine the type of cobordism and assign one of the following five algebraic maps.

Each map is applied only to the affected components and extended via identity on the remaining factors. The possible local maps are as follows:

\begin{enumerate}[label=(\roman*)]
  \item \textbf{Split a circle into two circles} : $\Delta: V \to V \otimes V$
  \[
  \begin{aligned}
    \Delta(v_+) &= v_+ \otimes v_- + v_- \otimes v_+, \\
    \Delta(v_-) &= v_- \otimes v_-.
  \end{aligned}
  \]

  \item \textbf{Merge two circles into a single circle} : $m: V \otimes V \to V$
  \[
  \begin{aligned}
    m(v_+ \otimes v_+) &= v_+, \\
    m(v_+ \otimes v_-) &= m(v_- \otimes v_+) = v_-, \\
    m(v_- \otimes v_-) &= 0.
  \end{aligned}
  \]

  \item \textbf{Saddle between arcs}: $W \otimes W \to W \otimes W$
  \[
  w \otimes w \mapsto 0.
  \]

  \item \textbf{Split an arc into an arc and a circle}: $W \to W \otimes V$
  \[
  w \mapsto w \otimes v_-.
  \]

  \item \textbf{Merge of arc and circle into a single arc}: $W \otimes V \to W$
  \begin{align*}
    w \otimes v_+ &\mapsto w, \\
    w \otimes v_- &\mapsto 0.
  \end{align*}
\end{enumerate}

%degree shift
The global differential in height $k$ is then defined by summing over all such adjacent pairs
\[
d^k = \sum_{\substack{s \to s' \\ |s| = k,\ |s'| = k+1}} (-1)^{\operatorname{sgn}(s \to s')} \cdot d_{s \to s'},
\]
where the sign $\operatorname{sgn}(s \to s')$ is the number of $1$s before the position that the changed entry in $s$.

\subsection{Generating Local Maps}

To construct the differentials in the Khovanov complex, we must compute the induced map between smoothing states $s \to s'$ that differ at a single crossing. This local transition determines a cobordism, which we analyze to produce the corresponding linear map between chain group generators.

\vspace{1ex}
\begin{example}
We implement a function \texttt{generate\_local\_map} to computes the local transition data. Here, the \texttt{smoothing\_state\_generator} is inherited from section \ref{section:example}.
\end{example}

\noindent\textbf{Iutput 1:}
\begin{verbatim}
local_map_1 = generate_local_map(smoothing_state_generator, "0000", "0001")
print_local_map(local_map_1)
\end{verbatim}

\vspace{1ex}
\noindent\textbf{Output 1:}

\begin{lstlisting}
Transition Type: saddle

Pre-State Elements:
  - Type: arc, Representative: [`3',`10|',`6|']
  - Type: circle, Representative: [`4',`9']
  - Type: arc, Representative: [`|1',`8',`2',`|7',`5']

Post-State Elements:
  - Type: arc, Representative: [`3',`10|',`|1',`8',`5']
  - Type: circle, Representative: [`4',`9']
  - Type: arc, Representative: [`2',`|7',`6|']

Calculated Coefficients:
  w([`3',`10|',`6|']) $\otimes$ v_+([`4',`9']) $\otimes$ w([`|1',`8',`2',`|7',`5'])
    $\rightarrow$ w([`3',`10|',`|1',`8',`5']) $\otimes$ v_+([`4',`9']) $\otimes$ w([`2',`|7',`6|']) : Coefficient = 0

  w([`3',`10|',`6|']) $\otimes$ v_+([`4',`9']) $\otimes$ w([`|1',`8',`2',`|7',`5'])
    $\rightarrow$ w([`3',`10|',`|1',`8',`5']) $\otimes$ v_-([`4',`9']) $\otimes$ w([`2',`|7',`6|']) : Coefficient = 0

  w([`3',`10|',`6|']) $\otimes$ v_-([`4',`9']) $\otimes$ w([`|1',`8',`2',`|7',`5'])
    $\rightarrow$ w([`3',`10|',`|1',`8',`5']) $\otimes$ v_+([`4',`9']) $\otimes$ w([`2',`|7',`6|']) : Coefficient = 0

  w([`3',`10|',`6|']) $\otimes$ v_-([`4',`9']) $\otimes$ w([`|1',`8',`2',`|7',`5'])
    $\rightarrow$ w([`3',`10|',`|1',`8',`5']) $\otimes$ v_-([`4',`9']) $\otimes$ w([`2',`|7',`6|']) : Coefficient = 0
\end{lstlisting}
\vspace{1ex}
\noindent\textbf{Input 2:}
\begin{verbatim}
local_map_2 = generate_local_map(smoothing_state_generator, "0000", "0100")
print_local_map(local_map_2)
\end{verbatim}

\vspace{1ex}
\noindent\textbf{Output 2:}

\begin{lstlisting}
Transition Type: split_arc_circle

Pre-State Elements:
  - Type: arc, Representative: [`3',`10|',`6|']
  - Type: circle, Representative: [`4',`9']
  - Type: arc, Representative: [`|1',`8',`2',`|7',`5']

Post-State Elements:
  - Type: arc, Representative: [`3',`10|',`6|']
  - Type: circle, Representative: [`4',`9']
  - Type: arc, Representative: [`|1',`|7']
  - Type: circle, Representative: [`8',`2',`5']

Calculated Coefficients:
  w([`3',`10|',`6|']) $\otimes$ v_+([`4',`9']) $\otimes$ w([`|1',`8',`2',`|7',`5']) $\to$ w([`3',`10|',`6|']) $\otimes$ v_+([`4',`9']) $\otimes$ w([`|1',`|7']) $\otimes$ v_+([`8',`2',`5']) : Coefficient = 0
  w([`3',`10|',`6|']) $\otimes$ v_+([`4',`9']) $\otimes$ w([`|1',`8',`2',`|7',`5']) $\to$ w([`3',`10|',`6|']) $\otimes$ v_+([`4',`9']) $\otimes$ w([`|1',`|7']) $\otimes$ v_-([`8',`2',`5']) : Coefficient = 1
  w([`3',`10|',`6|']) $\otimes$ v_+([`4',`9']) $\otimes$ w([`|1',`8',`2',`|7',`5']) $\to$ w([`3',`10|',`6|']) $\otimes$ v_-([`4',`9']) $\otimes$ w([`|1',`|7']) $\otimes$ v_+([`8',`2',`5']) : Coefficient = 0
  w([`3',`10|',`6|']) $\otimes$ v_+([`4',`9']) $\otimes$ w([`|1',`8',`2',`|7',`5']) $\to$ w([`3',`10|',`6|']) $\otimes$ v_-([`4',`9']) $\otimes$ w([`|1',`|7']) $\otimes$ v_-([`8',`2',`5']) : Coefficient = 0
  w([`3',`10|',`6|']) $\otimes$ v_-([`4',`9']) $\otimes$ w([`|1',`8',`2',`|7',`5']) $\to$ w([`3',`10|',`6|']) $\otimes$ v_+([`4',`9']) $\otimes$ w([`|1',`|7']) $\otimes$ v_+([`8',`2',`5']) : Coefficient = 0
  w([`3',`10|',`6|']) $\otimes$ v_-([`4',`9']) $\otimes$ w([`|1',`8',`2',`|7',`5']) $\to$ w([`3',`10|',`6|']) $\otimes$ v_+([`4',`9']) $\otimes$ w([`|1',`|7']) $\otimes$ v_-([`8',`2',`5']) : Coefficient = 0
  w([`3',`10|',`6|']) $\otimes$ v_-([`4',`9']) $\otimes$ w([`|1',`8',`2',`|7',`5']) $\to$ w([`3',`10|',`6|']) $\otimes$ v_-([`4',`9']) $\otimes$ w([`|1',`|7']) $\otimes$ v_+([`8',`2',`5']) : Coefficient = 0
  w([`3',`10|',`6|']) $\otimes$ v_-([`4',`9']) $\otimes$ w([`|1',`8',`2',`|7',`5']) $\to$ w([`3',`10|',`6|']) $\otimes$ v_-([`4',`9']) $\otimes$ w([`|1',`|7']) $\otimes$ v_-([`8',`2',`5']) : Coefficient = 1
\end{lstlisting}

This function thus automates the determination and evaluation of all local maps $d_{s \to s'}$ in the complex.

Each local map $d_{s \to s'}$ can be represented as a matrix acting between the corresponding basis elements of $C^i$ and $C^{i+1}$, denoted as $A_{i,j}$, where $i$ and $j$ are the indices of the smoothing states $s \in C^i$ and $s' \in C^{i+1}$, respectively. If no such map $d_{s \to s'}$ exists, the corresponding matrix $A_{i,j}$ is defined to be the zero matrix. Each nonzero matrix $A_{i,j}$ is further multiplied by a sign $(-1)^{\deg(s \to s')}$.

These local matrices are assembled into a global boundary matrix by placing each
$A_{i,j}$ into the corresponding block position of the full operator:
\[
A^i =
\begin{bmatrix}
(-1)^{\deg_{1,1}} A_{1,1} & (-1)^{\deg_{1,2}} A_{1,2} & \cdots \\
(-1)^{\deg_{2,1}} A_{2,1} & (-1)^{\deg_{2,2}} A_{2,2} & \cdots \\
\vdots  & \vdots  & \ddots
\end{bmatrix},
\]
where each block $A_{i,j}$ corresponds to the local map from smoothing state $s_i$ to $s_j'$, and $\deg_{k,l} = \deg(s_k \to s_l')$.

\subsection{Computing homology and quantum grading}

With the complete set of boundary matrices $A^i$ constructed from local maps, we can now compute the Khovanov homology. The Khovanov homology groups are defined as the quotient
\[
H^i = \ker(d^i) / \operatorname{im}(d^{i-1}).
\]
Algorithmically, we calculate the generators of $H^i$ by computing the left null space of the augmented matrix
\[
\begin{bmatrix}
A^{i} \\
(A^{i-1})^{T}
\end{bmatrix},
\]
where $(A^{i-1})^{T}$ is the transpose matrix of $A^{i-1}$.

Each generator of the homology group, obtained from the null space of the matrix is a linear combination of basis elements in the chain group $C^i$. Theoretically, since the differential preserves quantum grading, all basis elements contributing to a single homology generator should have the same quantum grading. This allows us, in principle, to assign a well-defined quantum grading to each generator.

However, in practice, we compute the null space using numerical methods such as Singular Value Decomposition (SVD), which may yield approximate (non-sparse) solutions. As a result, a generator may be expressed as a weighted sum of multiple basis elements with different quantum gradings.

To approximate the quantum grading of such a generator, we use the following formula:
\[
q = \frac{|x_1| q_1 + |x_2| q_2 + \cdots + |x_s| q_s}{|x_1| + |x_2| + \cdots + |x_s|},
\]
where $(x_1, x_2, \dots, x_s)$ are the coefficients in the generator vector, and $(q_1, q_2, \dots, q_s)$ are the quantum gradings of the corresponding basis elements.

\subsubsection{Example Usage: Computing Khovanov Homology}
Our code allows users to compute the Khovanov homology directly from a given planar diagram code representation of a knot or tangle. Here, we use our code to calculate the tangle in Example \ref{example:tritangle}:

\begin{verbatim}
from KHomology import calculate_Khovanov_homology_from_PD, explain_result

pdcode = [["2","5","3","6"],
         ["4","|1","5","2"],
         ["6","3","7|","4"]]
result = calculate_Khovanov_homology_from_PD(pdcode)
explain_result(result)
\end{verbatim}

The output will report the detected homology classes with their corresponding homological and quantum gradings:

\begin{verbatim}
Detect a homology class of dimension 0 with quantum degree 1.0.
Detect a homology class of dimension 2 with quantum degree 5.0.
Detect a homology class of dimension 3 with quantum degree 7.0.
\end{verbatim}

\section*{Code Availability}
The source code of this study is openly available at: \url{https://github.com/WeilabMSU/PKHT/}

\section*{Acknowledgments}
This work was supported in part by  Michigan State University Research Foundation and  Bristol-Myers Squibb  65109. Li was supported by an NITMB fellowship supported by grants from the NSF (DMS-2235451) and Simons Foundation (MP-TMPS-00005320). Jian was also supported by the Natural Science Foundation of China (NSFC Grant No. 12401080) and the start-up research fund from Chongqing University of Technology.

\bibliographystyle{plain}  % unsrtnat
\bibliography{Reference}

\end{document}